\documentclass[11pt]{amsart}
\usepackage[utf8]{inputenc}
\usepackage{amssymb, amsthm}
\usepackage[leqno]{amsmath}
\usepackage[foot]{amsaddr}
\pdfoutput=1 
\usepackage{enumerate}

\usepackage{prettyref,cite}
\usepackage[colorlinks,linkcolor=blue,citecolor=red]{hyperref}
\usepackage [a4paper, footskip=1cm, headheight = 16pt, top=3.7cm, bottom=3cm,  right=2.5cm,  left=2.5cm ]{geometry}
\newtheorem{theorem}{Theorem}
\newtheorem{lemma}[theorem]{Lemma}
\mathchardef\mh="2D

\newcounter{tbox}

\newcommand{\sta}[1]{\refstepcounter{tbox}\noindent{ \parbox{\textwidth}{\vspace*{0.3cm}(\thetbox) \emph{#1}\vspace*{0.3cm}}}}

\newcommand{\vsp}{\vspace*{3mm}}

\makeatletter
\newcommand{\otherlabel}[2]{\protected@edef\@currentlabel{#2}\label{#1}}
\makeatother

    \title[Complexity dichotomy for List-5-coloring]{Complexity dichotomy for List-5-coloring with a forbidden induced subgraph}
\author{Sepehr Hajebi$^{\ast}$}
\email{shajebi@uwaterloo.ca}
\author{Yanjia Li$^{\ast}$}
\email{yanjia.li@uwaterloo.ca}
\author{Sophie Spirkl$^{\ast \dagger}$}
\email{sspirkl@uwaterloo.ca}
\address{$^{\ast}$Department of Combinatorics and Optimization, University of Waterloo, Waterloo, Ontario N2L3G1, Canada}
\address{$^{\dagger}$ We acknowledge the support of the Natural Sciences and Engineering Research Council of Canada (NSERC), [funding reference number RGPIN-2020-03912]. Cette recherche a été financée par le Conseil de recherches en sciences naturelles et en génie du Canada (CRSNG), [numéro de référence RGPIN-2020-03912].}

\date{\today}

\allowdisplaybreaks

\begin{document}
\maketitle
\begin{abstract}
	For a positive integer $r$ and graphs $G$ and $H$, we denote by $G+H$ the disjoint union of $G$ and $H$ and by $rH$ the union of $r$ mutually disjoint copies of $H$. Also, we say $G$ is $H$\textit{-free} if $H$ is not isomorphic to an induced subgraph of $G$. We use $P_t$ to denote the path on $t$ vertices. For a fixed positive integer $k$, the \textsc{List-$k$-Coloring Problem} is to decide, given a graph $G$ and a list $L(v)\subseteq \{1,\ldots,k\}$ of colors assigned to each vertex $v$ of $G$, whether $G$ admits a proper coloring $\phi$ with $\phi(v)\in L(v)$ for every vertex $v$ of $G$, and the \textsc{$k$-Coloring Problem} is the \textsc{List-$k$-Coloring Problem} restricted to instances with $L(v)=\{1,\ldots, k\}$ for every vertex $v$ of $G$. We prove that for every positive integer $r$, the \textsc{List-$5$-Coloring Problem} restricted to $rP_3$-free graphs can be solved in polynomial time. Together with known results, this gives a complete dichotomy for the complexity of the \textsc{List-$5$-Coloring Problem} restricted to $H$-free graphs: For every graph $H$, assuming \textsf{P}$\neq$\textsf{NP}, the \textsc{List-$5$-Coloring Problem} restricted to $H$-free graphs can be solved in polynomial time if and only if $H$ is an induced subgraph of either $rP_3$ or $P_5+rP_1$ for some positive integer $r$. As a hardness counterpart, we also show that the \textsc{$k$-Coloring Problem} restricted to $rP_4$-free graphs is \textsf{NP}-complete for all $k\geq 5$ and $r\geq 2$. 
\end{abstract} 

\section{Introduction}
    Throughout this paper, all graphs are finite and simple.  We denote the set of positive integers by $\mathbb{N}$, and for every $k\in \mathbb{N}$, we define $[k]=\{1,\ldots, k\}$. Let $G$ be a graph. We denote by $V(G)$ and $E(G)$ the vertex set and the edge set of $G$, respectively. By a \textit{clique} in $G$ we mean a set of pairwise adjacent vertices, and a \textit{stable set} in $G$ is  a set of pairwise nonadjacent vertices. For every $d\in \mathbb{N}$ and every vertex $v\in V(G)$, we denote by $N^d_G(v)$ the set of all vertices in $G$ at distance $d$ from $v$, and by $N^d_G[v]$ the set of all vertices in $G$ at distance at most $d$ from $v$. In particular, we write $N_G(v)$ for $N_G^1(v)$, which is the set of neighbors of $v$ in $G$, and $N_G[v]$ for $N^1_G[v]=N_G(v)\cup \{v\}$. Also, for every $X\subseteq V(G)$, we define $N^d_G[X]=\bigcup_{x\in X}N^d_G[x]$ and $N^d_G(X)=N^d_G[X]\setminus X$. Again, we write $N_G(X)$ for $N_G^1(X)$ and $N_G[X]$ for $N^1_G[X]$. For every $Z$ which is either a vertex or a subset of vertices of $G$, we write $G-Z$ for the graph obtained from $G$ by removing $Z$. A graph $H$ is an \emph{induced subgraph} of a graph $G$ if $H$ is isomorphic to $G-X$ for some $X\subseteq V(G)$, and otherwise $G$ is \emph{$H$-free}. Also, for every graph $G$ and every $X\subseteq V(G)$, we denote \textit{the subgraph of $G$ induced on $X$}, that is, $G-(V(G)\setminus X)$, by $G|X$. For $r\in \mathbb{N}$ and graphs $G$ and $H$, we denote by $G+H$ the disjoint union of $G$ and $H$ and by $rH$ the union of $r$ pairwise disjoint copies of $H$. For all $t\in \mathbb{N}$, we use $P_t$ to denote the path on $t$ vertices. 

    Let $G$ be a graph and $k\in \mathbb{N}$. By a $k$\textit{-coloring} of $G$, we mean a function $\phi: V(G) \rightarrow [k]$. A coloring $\phi$ of $G$ is said to be \textit{proper} if $\phi(u) \neq \phi(v)$ for every edge $uv \in E(G)$. In other words, $\phi$ is proper if and only if for every $i\in [k]$, $\phi^{-1}(i)$ is a stable set in $G$. We say $G$ is \emph{$k$-colorable} if $G$ has a proper $k$-coloring. For fixed $k\in \mathbb{N}$, the \textsc{$k$-Coloring Problem} asks, given graph $G$, whether $G$ is $k$-colorable. 
 
    A \emph{$k$-list-assignment} of $G$ is a map $L:V(G)\rightarrow 2^{[k]}$. For $v\in V(G)$, we refer to $L(v)$ as the \textit{list of} $v$. Also, for every $i\in [k]$, we define $L^{(i)}=\{v\in V(G):i\in L(v)\}$. An \emph{$L$-coloring} of $G$ is a proper $k$-coloring $\phi$ of $G$ with $\phi(v) \in L(v)$ for all $v \in V(G)$. For example, if $L(v)=\emptyset$ for some $v\in V(G)$, then $G$ admits no $L$-coloring. Also, if $V(G)=\emptyset$, then $G$ vacuously admits an $L$-coloring for every $k$-list-assignment $L$.
    For fixed $k\in \mathbb{N}$, the \textsc{List-$k$-Coloring Problem} is to decide, given an instance $(G,L)$ consisting of a graph $G$ and a $k$-list-assignment $L$ for $G$, whether $G$ admits an $L$-coloring. Note that the \textit{$k$-coloring problem} is in fact the \textsc{List-$k$-Coloring Problem} restricted to instances $(G,L)$ where $L(v)=[k]$ for every $v\in V(G)$. 

    The \textsc{$k$-coloring Problem}, and so the \textsc{List-$k$-Coloring Problem}, are well-known to be \textsf{NP}-complete for all $k\geq 3$ \cite{karp}. This motivates studying the complexity of these problems restricted to graphs with a fixed forbidden induced subgraph, that is, $H$-free graphs for some fixed graph $H$. As a narrowing start, the following two theorems show that there is virtually no hope for a polynomial-time algorithm unless $H$ is a disjoint union of paths.

\begin{theorem}[Kami\'{n}ski and Lozin \cite{CycleFree}] \label{thm:cycle}
	For all $k\geq 3$, the \textsc{$k$-Coloring} problem restricted to $H$-free graphs is \textsf{NP}-complete if $H$ contains a cycle. 
\end{theorem}

\begin{theorem}[Holyer \cite{ClawFree}] \label{thm:claw}
	For all $k\geq 3$, the \textsc{$k$-Coloring Problem} restricted to $H$-free graphs is \textsf{NP}-complete if $H$ contains a `claw' (a vertex with three pairwise nonadjacent neighbors). 
\end{theorem}

Accordingly, an extensive body of work has been devoted to show that excluding certain paths (or their disjoint unions) makes the \textsc{$k$-Coloring} and the \textsc{List-$k$-Coloring} problem easier. Here is a list of known results in this direction.

\begin{theorem}\label{knownpoly}
		The \textsc{$k$-Coloring Problem} restricted to $H$-free graphs can be solved in polynomial time if:
		\begin{itemize}
			\item $H=P_6$ for $k=4$ \textup{(Chudnovsky, Spirkl and Zhong \cite{4P6})}; 
			\item $H=rP_2$ for all fixed $k,r\in \mathbb{N}$ \textup{(Golovach, Johnson, Paulusma and Song \cite{golovachsurvey})};
			\item $H=rP_3$ for $k=3$ and all fixed $r \in \mathbb{N}$ \textup{(Broersma, Golovach, Paulusma and Song \cite{3*P2+P4&3*rP3})};
		\end{itemize}
 		and the \textsc{List-$k$-Coloring Problem} restricted to $H$-free graphs can be solved in polynomial time if:
		\begin{itemize}
			\item $H=P_5$ for all fixed $k\in \mathbb{N}$ \textup{(Ho{\`a}ng, Kami{\'n}ski, Lozin, Sawada and Shu \cite{kP5})};
			\item $H=P_7$ for $k=3$ \textup{(Bonomo, Chudnovsky, Maceli, Schaudt, Stein and Zhong \cite{3P7})};
			\item $H=P_6+rP_3$ for $k=3$ and all $r\in \mathbb{N}$ \textup{(Chudnovsky, Huang, Spirkl and Zhong \cite{L3P6+rP3})};
			\item $H=P_5+rP_1$ for all $r, k\in \mathbb{N}$ \textup{(Couturier, Golovach, Kratsch and Paulusma \cite{LkP5+rP1&L5P4+P2-NP}). \label{thm:p5rp1}}
		\end{itemize}
	\end{theorem}
	
	On the other hand, the following hardness results are known. 	
	\begin{theorem}
		The \textsc{$k$-Coloring Problem} restricted to $H$-free graphs is \textsf{NP}-Complete if:
		\begin{itemize}
			\item $H=P_6$ for $k=5$, or $H=P_7$ for $k=4$ \textup{(Huang \cite{5P6&4P7-NP}); \label{thm:p6}}
			\item $H=P_5+P_2$ for $k=5$ \textup{(Chudnovsky, Huang, Spirkl and Zhong \cite{L3P6+rP3}); \label{thm:p5p2}}
		\end{itemize}
		and the \textsc{List-$k$-Coloring Problem} restricted to $H$-free graphs is NP-Complete if:
		\begin{itemize}
			\item $H=P_6$ for $k=4$ \textup{(Golovach, Paulusma and Song \cite{L4P6-NP})\label{thm:listp6}};
			\item $H=P_4+P_2$ for $k=5$ \textup{(Couturier, Golovach, Kratsch and Paulusma \cite{LkP5+rP1&L5P4+P2-NP}) \label{thm:p4p2}}.
		\end{itemize}
	\end{theorem}

Our main result is the following.

\begin{theorem} \label{thm:main}
	For every $r \in \mathbb{N}$, the \textsc{List-$5$-Coloring Problem} restricted to $rP_3$-free graphs can be solved in polynomial time. 
\end{theorem}

Note that in addition to extending the third bullet of Theorem \ref{knownpoly}, this completely classifies the complexity of the \textsc{List-$5$-Coloring Problem} restricted to $H$-free instances. Let us prove this formally:

\begin{theorem} \label{thm:dichotomy}
	Let $H$ be a graph. Assuming \textsf{P}$\neq$\textsf{NP}, the \textsc{List-$5$-Coloring Problem} restricted to $H$-free graphs can be solved in polynomial time if and only if $H$ is an induced subgraph of $rP_3$ or $P_5+rP_1$ for some $r \in \mathbb{N}$. 
\end{theorem}
\begin{proof}[Proof of Theorem \ref{thm:dichotomy} assuming Theorem \ref{thm:main}.]
	If $H$ is an induced subgraph of $rP_3$, then the result follows from Theorem \ref{thm:main}, and if $H$ is an induced subgraph of $P_5+rP_1$, then the result follows from Theorem \ref{thm:p5rp1}. So we may assume that neither is the case. 
	
	If $H$ is not a disjoint union of paths, then either $H$ is not a forest, in which case the result follows from Theorem \ref{thm:cycle}, or $H$ is a forest with a vertex of degree at least three, and so the result follows from Theorem \ref{thm:claw}. Therefore, we may assume that $H$ is a disjoint union of paths. Since $H$ is not an induced subgraph of $rP_3$, there is a connected component $C$ of  $H$ which is isomorphic to $P_t$ for some $t\geq 4$. If $t\geq 6$, then the result follows from the first bullet of Theorem \ref{thm:p6}. So we may assume that $t\in \{4,5\}$. Now, since $H$ is not an induced subgraph of $P_5+rP_1$, it follows that $H-V(C)$ contains an edge. But now $H$ contains $P_4+P_2$ as an induced subgraph, and thus the result follows from the fourth bullet of Theorem \ref{thm:p6}. This completes the proof.
\end{proof}

As a hardness counterpart to Theorem \ref{thm:main}, using a reduction similar to the one in \cite{LkP5+rP1&L5P4+P2-NP}, we also show that:
\begin{theorem} \label{thm:hardness}
	The \textsc{$k$-Coloring Problem} restricted to $2P_4$-free graphs (and hence $rP_4$-free graphs for every fixed $r\geq 2$) is \textsf{NP}-complete for all $k\geq 5$.
\end{theorem}
In view of Theorem \ref{thm:p5p2}, this leaves open the complexity of the \textsc{$5$-Coloring Problem} restricted to $H$-free graphs only if exactly one connected component $C$ of $H$ isomorphic to $P_4$, and $H-V(C)$ is an induced subgraph of $rP_3$ containing at least one edge for some $r\geq 1$.

The remainder of this paper is organized as follows. In Sections \ref{frugsec}-\ref{23sec}, we prepare the tools required for the proof of Theorem \ref{thm:main}. In Section \ref{mainsec}, we prove Theorem \ref{thm:main}, and finally in Section \ref{sec:hardness}, we prove Theorem \ref{thm:hardness}. It is worth noting that the main results of Sections \ref{frugsec} and \ref{goodsec}, namely Theorems \ref{frugality} and \ref{goodp3theorem}, respectively, are in fact proved for the \textsc{List-$k$-Coloring Problem} restricted to $rP_3$-free graphs with arbitrary $k$. However, our results from Section \ref{23sec} fail to extend to this general setting for $k\geq 6$. On the other hand, we were not able to decide whether there exists $k\in \mathbb{N}$ for which the \textsc{List-$k$-Coloring Problem} restricted to $rP_3$-free graphs is \textsf{NP}-hard. 

\section{Refinements, profiles and Frugality}\label{frugsec}
We begin with introducing the notions of a \textit{refinement} and a \textit{profile} as a unified terminology we employ pervasively in this paper. Let $k\in \mathbb{N}$  and  $(G,L)$ be an instance of the \textsc{List-$k$-Coloring Problem}. By a $(G,L)$\textit{-refinement} we mean an instance $(G',L')$ of the \textsc{List-$k$-Coloring Problem} where $G'$ is an induced subgraph of $G$ and $L'(v)\subseteq L(v)$ for all $v\in V(G')$. The $(G,L)$-refinement $(G',L')$ is \textit{spanning} if $G'=G$. Also, a $(G,L)$\textit{-profile} is a set of $(G,L)$-refinements. A $(G,L)$-profile is \textit{spanning} if all its elements are spanning. A large portion of this work deals with how the feasibility of an instance of the \textsc{List-$k$-Coloring Problem} is tied to the feasibility of certain refinements. For example, the following is easily observed.
\begin{lemma}\label{spanspread}
		Let $k\in \mathbb{N}$ be fixed and $(G,L)$ be an instance of the \textsc{List-$k$-Coloring Problem} and $(G,L')$ be a spanning $(G,L)$-refinement. If $G$ admits an $L'$-coloring, the $G$ admits an $L$-coloring.
\end{lemma}
Let $k\in \mathbb{N}$  and $(G,L)$ be an instance of the \textsc{List-$k$-Coloring Problem}. An $L$-coloring $\phi$ of $G$ is said to be \textit{frugal} if for all $v\in V(G)$ and every $i\in L(v)$, $v$ has at most one neighbor in $\phi^{-1}(i)$. This could be viewed as a list-variant of the so-called \textit{frugal coloring} introduced by Hind et al \cite{frugal}. Also, it is crucially different from another list-variant of frugal coloring studied in \cite{listfrugal}, where the restriction applies to all colors, not just those in the list of $v$. The following lemma is straightforward to verify.
\begin{lemma}\label{frugalspread}
	Let $k\in \mathbb{N}$ be fixed and $(G,L)$ be an instance of the \textsc{List-$k$-Coloring Problem}, $(G',L')$ be a $(G,L)$-refinement, and $\phi$ be a frugal $L$-coloring of $G$. If $\phi(v)\in L'(v)$ for every $v\in V(G')$ (that is, $\phi|_{V(G')}$ is an $L'$-coloring of $G'$), then it is a frugal $L'$-coloring of $G'$.
\end{lemma}
Note that for an instance $(G,L)$ of the \textsc{List-$k$-Coloring Problem}, if $|L(v)|=1$ for some $v\in V(G)$, then we may remove $v$ from $G$ and also remove the single color in $L(v)$ from the lists of all neighbors of $v$ in $G$, obtaining an instance with the same state of feasibility. To remain precise, let us state this simple observation formally, as follows. 
\begin{theorem}\label{kill1}
				Let $k\in \mathbb{N}$ be fixed and  $(G,L)$ be an instance of the \textsc{List-$k$-Coloring Problem}. Then there exists a $(G,L)$-refinement $(\hat{G},\hat{L})$ with the following specifications.
			\begin{itemize}
				\item $(\hat{G},\hat{L})$ can be computed from $(G,L)$ in time $\mathcal{O}(|V(G)|^2)$.
				\item $|\hat{L}(v)|\neq 1$ for all $v\in V(\hat{G})$.
				\item If $G$ admits a frugal $L$-coloring, then $\hat{G}$ admits a frugal $\hat{L}$-coloring. 
				\item If $\hat{G}$ admits an $\hat{L}$-coloring, then $G$ admits an $L$-coloring.
			\end{itemize}
		\end{theorem}
		\begin{proof}
			The proof is easy, so we only give a sketch and leave it to the reader to check the details. One can find in time $\mathcal{O}(|V(G)|)$ a vertex $v\in V(G)$ with  $|L(v)|=1$, or confirm that there is none. In the former case, we replace $G$ by $G-v$ and $L(w)$ by $L(w)\setminus L(v)$ for every vertex $w\in N_G(v)$. In the latter case, we output the current instance and stop. Applying the same procedure iteratively, it is straightforward to check that we obtain a $(G,L)$-refinement $(\hat{G},\hat{L})$ satisfying Theorem \ref{kill1}. This completes the proof. 
		\end{proof}
		The main goal of this section, though, is to establish a reduction from list-coloring to frugal list-coloring restricted to $rP_3$-free graphs. To achieve this, we need the main result of \cite{HGL}, the statement of which calls for a few definitions. A \textit{hypergraph} $H$ is an ordered pair $(V(H),E(H))$ where $V (H)$ is a finite set of vertices and $E(H)$ is a collection of nonempty subsets of $V (H)$, usually referred to as \textit{hyperedges}. A \textit{matching} in $H$ is a set of pairwise disjoint hyperedges, and a \textit{vertex-cover} in $H$ is a set of vertices meeting every hyperedge. We denote by $\nu(H)$ the maximum size of a matching in $H$, and by $\tau(H)$ the minimum size of a vertex-cover in $H$. Also, we denote by $\Lambda(H)$ the maximum $k\geq 2$ for which there exists hyperedges $e_1,...,e_k\in E(H)$ with the following property. For all distinct $i,j\in [k]$, there exists a vertex $v_{i,j}\in e_i\cap e_j$ which belongs to no other hyperedge among $e_1,\ldots, e_k$. If there is no such $k$ (that is, if the elements of $E(H)$ are mutually disjoint), then we set $\Lambda (H)=2$.

\begin{theorem}[Ding, Seymour and Winkler\cite{HGL}]\label{HGL}
	For every hypergraph $H$, we have
	\[\tau(H) \leq  11\Lambda (H)^2(\Lambda(H)+\nu(H)+3)\binom{\Lambda(H) + \nu(H) }{\nu(H)}^2\cdot\]
\end{theorem}
We apply Theorem  \ref{HGL} to prove the following.
\begin{lemma}\label{biplem}
		Let $r\in \mathbb{N}$ and $\eta(r)=11(r+1)^2(2r+3)\binom{2r}{r-1}$. Let $G$ be an $rP_3$-free graph and $A,B$ be two disjoint stable sets in $G$, such that every vertex in $B$ has at least two neighbors in $A$. Then there exists $S\subseteq A$ with $|S|\leq \eta(r)$ such that every vertex in $B$ has a neighbor in $S$.
		\end{lemma}
	\begin{proof}
		For every vertex in $b\in B$, let $A(b)=N_G(b)\cap A$. Let $H$ be the hypergraph with $V(H)=A$ and $E(H)=\{A(b): b\in B\}$.
		
	 \sta{\label{nubndd}We have $\nu(H)\leq r-1$.}
		
	Suppose not. Let $b_1,\ldots, b_r\in B$ be distinct such that hyperedges $A(b_1)\, \ldots, A(b_r)$ of $H$ are pairwise disjoint. By the assumption, for each $i\in [r]$, there exist two distinct vertices $a_i,c_i\in  A(b_i)$. But then $G|\{a_i,b_i,c_i:i\in [r]\}$ is isomorphic to $rP_3$, a contradiction. This proves \eqref{nubndd}.
		
		\sta{\label{lambndd} We have $\Lambda (H)\leq r+1$.}
		
	For otherwise $\Lambda (H)\geq r+2\geq 3$, and so there exist distinct vertices $b_1,\ldots, b_{r+2}\in B$ with the following property. For all distinct $i,j\in [r+2]$, there exists a vertex $c_{i,j}\in A(b_i)\cap A(b_j)$ which belongs to no other set among $A(b_1), \ldots, A(b_{r+2})$. But then $G|\{b_i,c_{i,r+1},c_{i,r+2}:i\in [r]\}$ is isomorphic to $rP_3$, a contradiction. This proves \eqref{lambndd}.\vsp
		
	From \eqref{nubndd} and \eqref{lambndd} combined with Theorem \ref{HGL}, we obtain $\tau(H)\leq 11(r+1)^2(2r+3)\binom{2r}{r-1}=\eta(r)$, and so $H$ has a vertex-cover of size at most $\eta(r)$. In other words, there exists $S\subseteq A$ with $|S|\leq \eta(r)$ such that for every vertex $b\in B$, $S\cap A(b)\neq \emptyset$; that is, every vertex in $B$ has a neighbor in $S$. This completes the proof of Lemma \ref{biplem}. 
	\end{proof}
Now we can prove the main theorem of this section.
	\begin{theorem}\label{frugality}
		For all fixed $k,r\in \mathbb{N}$, there exists $\pi(k,r)\in \mathbb{N}$ with the following property. Let $(G,L)$ be an instance of the \textsc{List-$k$-Coloring Problem} where $G$ is $rP_3$-free. Then there exists a spanning $(G,L)$-profile $\Pi(G,L)$ with the following specifications.
		\begin{itemize}
			 \item \sloppy $|\Pi(G,L)|\leq \mathcal{O}(|V(G)|^{\pi(k,r)})$ and $\Pi(G,L)$ can be computed from $(G,L)$ in time $ \mathcal{O}(|V(G)|^{\pi(k,r)})$.
			\item If $G$ admits an $L$-coloring, then for some $(G,L')\in \Pi(G,L)$, $G$ admits a frugal $L'$-coloring.
		\end{itemize}
	\end{theorem}
	\begin{proof}
	Let $\eta(r)$ be as in Lemma \ref{biplem}. Let $\mathcal{S}$ be the set of all $k$-tuples $(S_1,\ldots, S_k)$ of subsets of $V(G)$ where
		\begin{itemize}
			\item[(S1)]  $S_i\subseteq L^{(i)}$ and $|S_i|\leq (k-1)\eta(r)$ for all $i\in [k]$;
			\item[(S2)] $S_i$ is a stable set; and
			\item[(S3)] $S_i\cap S_j=\emptyset$ for all distinct $i,j\in [k]$.	
	\end{itemize}
		For each $S=(S_1,\ldots, S_k)\in \mathcal{S}$, we define a $k$-list-assignment $L_S$ of $G$ as follows. Let $v \in V(G)$.	
		\begin{itemize}
			\item[(L1)] If $v\in S_i$ for some  $i\in [k]$, then let $L_S(v)=\{i\}$.
			\item[(L2)] Otherwise, if $v\in V(G)\setminus (\bigcup_{i=1}^kS_i)$, then let $L_S(v)=L(v)\setminus \{i\in [k]: N_G(v)\cap S_i\neq \emptyset\}$.
		\end{itemize}
	
	 This definition immediately yields the following.\\
		
		\sta{\label{inotin}For all $S=(S_1,\ldots, S_k)\in \mathcal{S}$, $i\in [k]$ and $v\in V(G)\setminus S_i$ with a neighbor in $S_i$, we have $i\notin L_S(v)$.}
		
		Note that for every $S\in \mathcal{S}$, $(G,L_S)$ is a spanning $(G,L)$-refinement. Consider the spanning $(G,L)$-profile $\Pi(G,L)=\{(G,L_S):S\in \mathcal{S}\}$.\\
		
		\sloppy\sta{\label{frugrun}$|\Pi(G,L)|\leq \mathcal{O}(|V(G)|^{k(k-1)\eta(r)})$ and $\Pi(G,L)$ can be computed from $(G,L)$ in time $ \mathcal{O}(|V(G)|^{k(k-1)\eta(r)+2})$.}
		
	Let $\mathcal{T}$ be the set of all $k$-tuples $(S_1,\ldots, S_k)$ of subsets of $V(G)$ satisfying (S1). Then for each $i\in [k]$, there are at most $((k-1)\eta(r)+1)|V(G)|^{(k-1)\eta(r)}$ possibilities for $S_i$. As a result, we have  $|\mathcal{T}|\leq ((k-1)\eta(r)+1)^k|V(G)|^{k(k-1)\eta(r)}$, which along with $|\Pi(G,L)|=|\mathcal{S}|$ and $\mathcal{S}\subseteq \mathcal{T}$ proves the first assertion. For the second, it is straightforward to observe that the elements of $\mathcal{T}$ can be enumerated in time $\mathcal{O}(|\mathcal{T}|)$. Then, for each $S=(S_1,\ldots, S_k)\in \mathcal{T}$, one can check in constant time whether $S$ satisfies (S2) and (S3). Thus, $\mathcal{S}$ can be computed in time $\mathcal{O}(|\mathcal{T}|)$. Also, for each $S\in \mathcal{S}$ and every $v\in V(G)$, it is readily seen from (L1) and (L2) that $L_S(v)$ can be computed in time $\mathcal{O}(|V(G)|)$. Therefore, $\Pi(G,L)$ can be computed from $(G,L)$ in time $\mathcal{O}(|\mathcal{T}||V(G)|^2)= \mathcal{O}(|V(G)|^{k(k-1)\eta(r)+2})$. This proves \eqref{frugrun}.\vsp

	\sta{\label{applybiplem}Let $\phi$ be an $L$-coloring of $G$ and $i,j\in [k]$ be distinct. Let $B_{i,j}$ be the set of all vertices in $\phi^{-1}(j)$ with at least two neighbors in $\phi^{-1}(i)$. Then there exists $S_{i,j}\subseteq \phi^{-1}(i)$ with $|S_{i,j}|\leq \eta(r)$ such that every vertex in $B_{i,j}$ has a neighbor in $S_{i,j}$.}	
	
	A direct application of Lemma \ref{biplem} to $G$ and stable sets $A=\phi^{-1}(i)$ and $B=B_{i,j}$ proves \eqref{applybiplem}.
		
		\sta{\label{ordtofrug}If $G$ admits an $L$-coloring, then for some $S\in \mathcal{S}$, $G$ admits a frugal $L_S$-coloring.}
		Let $\phi$ be an $L$-coloring of $G$. For distinct $i,j\in [k]$, let $B_{i,j}$ be the set of all vertices in $\phi^{-1}(j)$ with at least two neighbors in $\phi^{-1}(i)$. By \eqref{applybiplem}, there exists $S_{i,j}\subseteq \phi^{-1}(i)$ with $|S_{i,j}|\leq \eta(r)$ such that every vertex in $B_{i,j}$ has a neighbor in $S_{i,j}$. For each $i\in [k]$, let $S_i=\bigcup_{j\in [k],j\neq i}S_{i,j}$. Then from \eqref{applybiplem}, we have $S_i\subseteq \phi^{-1}(i)\subseteq L^{(i)}$, $|S_i|\leq (k-1)\eta(r)$ and  $S_i\cap S_j\subseteq \phi^{-1}(i)\cap \phi^{-1}(j)= \emptyset$ for every $j\in [k]$. It follows that $S=(S_1,\ldots,S_k)$ satisfies both (S1), (S2) and (S3), and so $S\in \mathcal{S}$. Let $L_S$ be the corresponding list-assignment defined in (L1) and (L2). We claim that $\phi$ is a frugal $L_S$-coloring of $G$. To see this, note that being an $L$-coloring, $\phi$ is proper. In addition, for every $v\in V(G)$, if $v\in S_i\subseteq \phi^{-1}(i)$ for some  $i\in [k]$, then by (L1), we have $L_S(v)=\{i\}=\{\phi(v)\} $. Otherwise, if $v\in V(G)\setminus (\bigcup_{i=1}^k S_i)$, then since $\phi^{-1}(\phi(v))$ is a stable set of $G$ containing $S_{\phi(v)}\cup \{v\}$, $v$ has no neighbor in $S_{\phi(v)}$, and so by (L2), we have $\phi(v)\in L_S$. As a result, we have  $\phi(v)\in L_S$ for every $v\in V(G)$, and $\phi$ is in fact an $L_S$-coloring.
	It remains to argue the frugality of $\phi$. For each $i\in [k]$, let $B_i=\bigcup_{j\in [k], j\neq i}B_{i,j}$; that is, $B_i$ is the set of all vertices in $G$ with at least two neighbors in $\phi^{-1}(i)$. Note that $B_i\subseteq V(G)\setminus S_i$.  By \eqref{applybiplem}, every vertex in $B_i$ has a neighbor in $S_i$. Therefore, by \eqref{inotin}, we have $i\notin L_S(v)$ for every $v\in B_i$. In other words, for all $v\in V(G)$ and every $i\in L_S(v)$, $v$ has at most one neighbor in $\phi^{-1}(i)$, and so $\phi$ is frugal. This proves \eqref{ordtofrug}.\vsp

Finally, by setting $\pi(k,r)=k(k-1)\eta(r)+2$, from \eqref{frugrun} and \eqref{ordtofrug}, we conclude that $\Pi(G,L)$ satisfies Theorem \ref{frugality}. This completes the proof.
\end{proof}
\section{Good $P_3$'s}\label{goodsec}
Let $G$ be a graph and $\{x_1,x_2,x_3\}\subseteq V(G)$ with $E(G|\{x_1,x_2,x_3\})=\{x_1x_2,x_2x_3\}$. Then $G|\{x_1,x_2,x_3\}$ is isomorphic to $P_3$, and we refer to it as an \textit{induced $P_3$} in $G$, denoting it by $x_1-x_2-x_3$. Also, for all $Z,W\subseteq V(G)$, we say $Z$ is \textit{complete} (\textit{anticomplete}) to $W$ if $Z\cap W=\emptyset$ and every vertex in $Z$ is adjacent (nonadjacent) to every vertex in $W$. If $Z=\{z\}$ and $Z$ is complete (anticomplete) to $W$, then we say $z$ is \textit{complete} (\textit{anticomplete}) to $W$. For two induced $P_3$'s $P$ and $Q$ in $G$, we say $P$ is \textit{anticomplete} to $Q$ (or \textit{$P$ and $Q$ are anticomplete}), if their vertex sets are anticomplete in $G$.

Let $k\in \mathbb{N}$ be an integer and $\gamma=(I_1,I_2,I_3)$ be a triple of subsets of $[k]$. We say $\gamma$ is \textit{good} if $|I_1|,|I_2|,|I_3|\geq 2$ and $I_1\cap I_2$, $I_1\cap I_3$ and $I_2\cap I_3$ are all nonempty. Let $(G,L)$ be an instance of the \textsc{List-$k$-Coloring Problem} and $x_1-x_2-x_3$ be an induced $P_3$ in $G$. We refer to $(L(x_1),L(x_2),L(x_3))$ as the $L$\textit{-type} of $x_1-x_2-x_3$ and to $|L(x_1)|+|L(x_2)|+|L(x_3)|$ as the $L$\textit{-weight} of $x_1-x_2-x_3$. Also,  an $L$\textit{-good} $P_3$ in $G$ is an induced $P_3 $ with good $L$-type. The following lemma, easy to check, asserts that excluding good $P_3$'s is inherited by refinements.
\begin{lemma}\label{goodspread}
	Let $(G,L)$ be an instance of the \textsc{List-$k$-Coloring Problem} and $(G',L')$ be a $(G,L)$-refinement. Suppose that $G$ has no $L$-good $P_3$. Then $G'$ has no $L'$-good $P_3$.
\end{lemma}
In this section, we show how to reduce an $rP_3$-free instance of the \textsc{List-$k$-Coloring Problem} to polynomially many instances with no good $P_3$ in polynomial time. This goal is attained in Theorem \ref{goodp3theorem}. First, we need two lemmas.
\begin{lemma}\label{goodp3lem1}
	Let $k\in \mathbb{N}$ be fixed and $\gamma=(I_1,I_2,I_3)$ be a good triple of subsets of $[k]$. Let $(G,L)$ be an instance of the \textsc{List-$k$-Coloring Problem} and $x_1-x_2-x_3$ be an induced $P_3$ in $G$ of $L$-type $(I_1,I_2,I_3)$. Then there exists a spanning $(G,L)$-profile  $\Upsilon_1(G,L)$ with the following specifications.
	\begin{itemize}
		\item \sloppy $|\Upsilon_1(G,L)|\leq \mathcal{O}(|V(G)|^{3k-3})$ and $\Upsilon_1(G,L)$ can be computed from $(G,L)$ in time $ \mathcal{O}(|V(G)|^{3k-2})$.
		\item For every $(G,L_1)\in \Upsilon_1(G,L)$, every induced $P_3$ in $G$ of $L_1$-type $(I_1,I_2,I_3)$ is anticomplete to (and thus disjoint from) $x_1-x_2-x_3$.
		\item If $G$ admits a frugal $L$-coloring, then for some $(G,L_1)\in \Upsilon_1(G,L)$, $G$ admits a frugal $L_1$-coloring. 
	\end{itemize}
\end{lemma}  
\begin{proof}
 Let $\mathcal{S}$ be the set of all pairs $(S,\psi)$ where
\begin{itemize}
	\item[(T1)] $S$ is a subset of $N_G[\{x_1,x_2,x_3\}]$ containing $\{x_1,x_2,x_3\}$ with $|S|\leq 3k$, and
	\item[(T2)] $\psi$ is an $L|_S$-coloring of $G|S$, where for every $i\in \{1,2,3\}$ and every $j\in L(x_i)$, $x_i$ has at most one neighbor $v$ in $S$ with $\psi(v)=j$.
\end{itemize}
We deduce:

\sta{\label{goodlem1Srun}$|\mathcal{S}|\leq \mathcal{O}(|V(G)|^{3k-3})$ and $\mathcal{S}$ can be computed from $(G,L)$ in time $\mathcal{O}(|V(G)|^{3k-3})$.}

Let $\mathcal{T}$ be the set of all pairs $(S,\psi)$, consisting of a set $S$  satisfying (T1) and a coloring $\psi:S\rightarrow[k]$ of $G|S$. Note that for each $(S,\psi)\in \mathcal{T}$, there are at most $(3k-2)|V(G)|^{3k-3}$ choices for $S$, and
for each such choice, there are $k^{|S|}\leq k^{3k}$ possibilities for $\psi$. So we have $|\mathcal{T}|\leq (3k-2)k^{3k}|V(G)|^{3k-3}$, which along with $\mathcal{S}\subseteq \mathcal{T}$  proves the first assertion. To see the second, note that the elements of $\mathcal{T}$ can be enumerated in time $\mathcal{O}(|\mathcal{T}|)=\mathcal{O}(|V(G)|^{3k-3})$. Also, for every $(S,\psi)\in \mathcal{T}$, since $|S|\leq 3k$, it can be checked in constant time whether $\psi$ satisfies (T2), or equivalently $(S,\psi)\in \mathcal{S}$.  Hence, $\mathcal{S}$ can be computed from $(G,L)$ in time $\mathcal{O}(|\mathcal{T}|)=\mathcal{O}(|V(G)|^{3k-3})$. This proves \eqref{goodlem1Srun}.\vsp

For every $\sigma=(S,\psi)\in \mathcal{S}$, consider the $k$-list-assignment $L_{\sigma}$ of $G$, defined as follows. Let $v\in V(G)$.
\begin{itemize}
	\item[(M1)] If $v\in S$, then $L_{\sigma}(v)=\{\psi(v)\}$.
	\item [(M2)] If $v\in N_G[\{x_1,x_2,x_3\}]\setminus S$, then $L_{\sigma}(v)=L(v)\setminus (\bigcup_{ j\in \{1,2,3\}, v\in N_G(x_j)} I_j)$.
	\item [(M3)] If $v\notin N_G[\{x_1,x_2,x_3\}]$, then $L_{\sigma}(v)=L(v)$. 
\end{itemize}
Note that for every $\sigma\in \mathcal{S}$, $(G,L_{\sigma})$ is a spanning $(G,L)$-refinement. Consider the spanning $(G,L)$-profile $\Upsilon_1(G,L)=\{(G,L_{\sigma}):\sigma\in \mathcal{S}\}$.\\

\sta{\label{goodlem1run} $|\Upsilon_1(G,L)|\leq \mathcal{O}(|V(G)|^{3k-3})$ and $\Upsilon_1(G,L)$ can be computed  from $(G,L)$ in time $ \mathcal{O}(|V(G)|^{3k-2})$.}

The first assertion follows from \eqref{goodlem1Srun} and the fact that $|\Upsilon_1(G,L)|=|\mathcal{S}|$. For the second, we need to compute $\mathcal{S}$, which by \eqref{goodlem1Srun} is attainable in time $\mathcal{O}(|V(G)|^{3k-3})$. Then, for every $\sigma\in \mathcal{S}$, it is easily seen from (M1),(M2) and (M3) that $L_{\sigma}$ can be computed in time $\mathcal{O}(|V(G)|)$. Therefore,  $\Upsilon_1(G,L)$ can be computed in time $\mathcal{O}(|V(G)|^{3k-3}+|V(G)||\mathcal{S}|)=\mathcal{O}(|V(G)|^{3k-2})$, where the last equality follows from \eqref{goodlem1Srun}. This proves \eqref{goodlem1run}.

\sta{\label{goodlem1anti}For every $\sigma=(S,\psi)\in \Upsilon_1(G,L)$, every induced $P_3$ in $G$ of $L_{\sigma}$-type $\gamma$ is anticomplete to $x_1-x_2-x_3$.}

Let $Q=y_1-y_2-y_3$ be an induced $P_3$ in $G$ and $(L_{\sigma}(y_1),L_{\sigma}(y_2),L_{\sigma}(y_3))=\gamma$. Since $\gamma$ is good, we have $|L_{\sigma}(y_i)|=|I_i|\geq 2$ for all $i\in \{1,2,3\}$, while $|L_{\sigma}(v)|=1$ for all $v\in S$. So $S\cap \{y_1,y_2,y_3\}=\emptyset$. Also, if $y_i\in N_G(x_j)\setminus S$ for some $i,j\in \{1,2,3\}$, then by (M2), we have $I_i=L_{\sigma}(y_i)\subseteq  L(y_i)\setminus I_j$, and so $I_i\cap I_j=\emptyset$, which violates $\gamma$ being good. Hence, $N_G[\{x_1,x_2,x_3\}]\cap \{y_1,y_2,y_3\}=\emptyset$, and so $Q$ is anticomplete to $x_1-x_2-x_3$. This proves \eqref{goodlem1anti}.

\sta{\label{goodlem1frugtofrug} If $G$ admits a frugal $L$-coloring, then for some $\sigma\in \mathcal{S}$, $G$ admits a frugal $L_{\sigma}$-coloring.}

Let $\phi$ be a frugal $L$-coloring of $G$. For every $j\in \{1,2,3\}$, we define $A_j=\{v\in N_G(x_j): \phi(v)\in I_j\}$. Let $S=\{x_1,x_2,x_3\}\cup A_1\cup A_2\cup A_3$. Then we have $\{x_1,x_2,x_3\}\subseteq S\subseteq N_G[\{x_1,x_2,x_3\}]$, and since $\phi$ is a frugal $L$-coloring of $G$, we have $|A_j|\leq |I_j|-1\leq k-1$ for every $i\in \{1,2,3\}$, which in turn implies that $|S|\leq 3k$. Thus, $S$ satisfies (T1). We define $\psi=\phi|_S$. Then, again by the frugality of $\phi$, for every $i\in \{1,2,3\}$ and every $j\in L(x_i)$, $x_i$ has at most one neighbor $v$ in $S$ with $\psi(v)=j$. In other words, $\psi$ satisfies (S2), and so $\sigma=(S,\psi)\in \mathcal{S}$. Let $L_{\sigma}$ be the corresponding $k$-list-assignment defined in (M1), (M2) and (M3). We claim that $\phi$ is a frugal $L_{\sigma}$-coloring of $G$. Let $v\in V(G)$. If $v\in S$, then by (M1), we have $\phi(v)=\psi(v)\in \{\psi(v)\}=L_{\sigma}(v)$. Also, if $v\in N_G(\{x_1,x_2,x_3\})\setminus S=N_G(\{x_1,x_2,x_3\})\setminus (A_1\cup A_2\cup A_3)$, then for every $j\in \{1,2,3\}$ with $v\in N_G(x_j)$, we have $\phi(v)\notin I_j$, and so by (M2), we have $\phi(v)\in L(v)\setminus (\bigcup_{ j\in \{1,2,3\}, v\in N_G(x_j)} I_j)=L_{\sigma}(v)$. Finally, if $v\notin N_G(\{x_1,x_2,x_3\})$, then by (M3), we have $\phi(v)\in L(v)=L_{\sigma}(v)$. In summary, we have $\phi(v)\in L_{\sigma}(v)$ for every $v\in V(G)$. Therefore, by Lemma \ref{frugalspread}, $\phi$ is a frugal $L_{\sigma}$-coloring of $G$. This proves \eqref{goodlem1frugtofrug}.\vsp

Finally, from \eqref{goodlem1run}, \eqref{goodlem1anti} and \eqref{goodlem1frugtofrug}, we conclude that $\Upsilon_1(G,L)$ satisfies Lemma \ref{goodp3lem1}. This completes the proof.
\end{proof}

Let $k\in \mathbb{N}$ and $\gamma= (I_1,I_2,I_3)$ be a triple of  subsets of $[k]$. Also, let $(G,L)$ be an instance of the \textsc{List-$k$-Coloring Problem}. We denote by $f_{L,\gamma}(G)$ be the maximum number of mutually anticomplete induced $P_3$'s in $G$ of $L$-type $\gamma$. 

\begin{lemma}\label{goodp3lem2}
	Let $k\in \mathbb{N}$ be fixed and $\gamma= (I_1,I_2,I_3)$ be a good triple of subsets of $[k]$. Let $(G,L)$ be an instance of the \textsc{List-$k$-Coloring Problem}. Also, suppose that no $L$-good $P_3$ in $G$ is of $L$-weight strictly larger than $|I_1|+|I_2|+|I_3|$. Then there exists a spanning $(G,L)$-profile $\Upsilon_2(G,L)$ with the following specifications.
	\begin{itemize}
		\item \sloppy $|\Upsilon_2(G,L)|\leq \mathcal{O}(|V(G)|^{(3k-3)f_{L,\gamma}(G)})$ and $\Upsilon_2(G,L)$ can be computed from $(G,L)$ in time $\mathcal{O}(|V(G)|^{(3k-2)f_{L,\gamma}(G)})$.
		\item For every $(G,L_2)\in \Upsilon_2(G,L)$, $G$ has no induced $P_3$ of $L_2$-type $\gamma$ (and, of course, none of $L_2$-weight strictly larger than $|I_1|+|I_2|+|I_3|$).
	\item If $G$ admits a frugal $L$-coloring, then for some $(G,L_2)\in \Upsilon_2(G,L)$, $G$ admits a frugal $L_2$-coloring. 
	\end{itemize}
\end{lemma}  
\begin{proof}
For fixed $\gamma$, we proceed by induction on $f_{L,\gamma}(G)$. If $f_{L,\gamma}(G)=0$, then $\Upsilon_2(G,L)=\{(G,L)\}$ satisfies Lemma \ref{goodp3lem2}. So we may assume that $k\geq 2$, and we may choose $x_1-x_2-x_3$ as an induced $P_3$ in $G$ with $(L(x_1),L(x_2),L(x_3))=\gamma$. Consequently, we may apply Lemma \ref{goodp3lem1} to $(G,L)$, $\gamma$ and $x_1-x_2-x_3$, obtaining a spanning $(G,L)$-profile  $\Upsilon_1(G,L)$ satisfying Lemma \ref{goodp3lem1}.

\sta{\label{fsmaller} For every $(G,L_1)\in \Upsilon_1(G,L)$, we have $f_{L_1,\gamma}(G)<f_{L,\gamma}(G)$.}

 Suppose for a contradiction that $f_{L_1,\gamma}(G)\geq f_{L,\gamma}(G)=t\geq 1$. We may consider a collection $\{q^i_1-q^i_2-q^i_3:i\in [t]\}$ of $t$ mutually anticomplete induced $P_3$'s in $G$, such that for every $i\in [t]$, $(L_1(q^i_1),L_1(q^i_2),L_1(q^i_3))=\gamma$. Now, for each $i\in [t]$, by the second bullet of Lemma \ref{goodp3lem1}, $q^i_1-q^i_2-q^i_3$ is anticomplete to $x_1-x_2-x_3$. Also, since $ (G,L_1) $ is a $(G,L)$-refinement,  for every $j\in \{1,2,3\}$, we have $I_j=L_1(q^i_j)\subseteq L(q^i_j)$. Therefore, $\gamma$ being good, $q^i_1-q^i_2-q^i_3$ is an $L$-good $P_3$ in $G$ of $L$-weight at least $|I_1|+|I_2|+|I_3|$. This, along with the assumption of Lemma \ref{goodp3lem2} that no $L$-good $P_3$ in $G$ is of $L$-weight strictly larger than $|I_1|+|I_2|+|I_3|$, implies that $L(q^i_j)=I_j$ for every $j\in \{1,2,3\}$. In other words, $q^i_1-q^i_2-q^i_3$ is an induced $P_3$ in $G$ of $L$-type $\gamma$ which is anticomplete to $x_1-x_2-x_3$. Hence, $\{q^i_1-q^i_2-q^i_3:i\in [t]\}\cup \{x_1-x_2-x_3\}$ comprises $t+1$ mutually anticomplete $P_3$'s in $G$ of $L$-type $\gamma$, which is impossible. This proves \eqref{fsmaller}.

\sta{\label{goodlem2IH}Let $(G,L_1)\in\Upsilon_1(G,L)$. Then there exists a spanning $(G,L_1)$-profile $\Upsilon_2(G,L_1)$ with the following specifications.
\begin{itemize}
	\item \sloppy $|\Upsilon_2(G,L_1)|\leq \mathcal{O}(|V(G)|^{(3k-3)(f_{L,\gamma}(G)-1)})$ and $\Upsilon_2(G,L_1)$ can be computed from $(G,L_1)$ in time $\mathcal{O}(|V(G)|^{(3k-2)(f_{L,\gamma}(G)-1)})$.
	\item For every $(G,L_2)\in \Upsilon_2(G,L_1)$, $G$ has no induced $P_3$ of $L_2$-type $\gamma$ (and, of course, none of $L_2$-weight strictly larger than $|I_1|+|I_2|+|I_3|$).
	\item If $G$ admits a frugal $L_1$-coloring, then for some $(G,L_2)\in \Upsilon_2(G,L)$, $G$ admits a frugal $L_2$-coloring. 
\end{itemize}}

By the assumption of Lemma \ref{goodp3lem2}, $G$ has no $L$-good $P_3$ of $L_1$-weight strictly larger than $|I_1|+|I_2|+|I_3|$. So since $(G,L_1)$ is a $(G,L)$-refinement,  $G$ has no $L_1$-good $P_3$ of $L_1$-weight strictly larger than $|I_1|+|I_2|+|I_3|$. This, along with \eqref{fsmaller} and the induction hypothesis, proves \eqref{goodlem2IH}.\vsp

\sloppy Finally, we define  $\Upsilon_2(G,L)=\bigcup_{(G,L_1)\in \Upsilon_1(G,L)}\Upsilon_2(G,L_1)$, where for every $(G,L_1)\in \Upsilon_1(G,L)$, $\Upsilon_2(G,L_1)$ is as promised in \eqref{goodlem2IH}. By the first bullet of \eqref{goodlem2IH} and the first bullet of Lemma \ref{goodp3lem1}, we have $|\Upsilon_2(G,L)|\leq \mathcal{O}\left(|V(G)|^{(3k-3)f_{L,\gamma}(G)}\right)$ and $\Upsilon_2(G,L)$ can be computed from $(G,L)$ in time $\mathcal{O}\left(|V(G)|^{(3k-2)}\right)+\mathcal{O}\left(|V(G)|^{(3k-3)}\right)\mathcal{O}\left(|V(G)|^{(3k-2)(f_{L,\gamma}(G)-1)}\right)=\mathcal{O}\left(|V(G)|^{(3k-2)f_{L,\gamma}(G)}\right)$. So $\Upsilon_2(G,L)$ satisfies the first bullet of Lemma \ref{goodp3lem2}. Also, the second bullet of Lemma \ref{goodp3lem2} for $\Upsilon_2(G,L)$ follows from the second bullet of \eqref{goodlem2IH}, and the third bullet of Lemma \ref{goodp3lem2} for $\Upsilon_2(G,L)$ follows from the third bullet of \eqref{goodlem2IH} together with the third bullet of Lemma \ref{goodp3lem1}. Hence, $\Upsilon_2(G,L)$ satisfies Lemma \ref{goodp3lem2}. This completes the proof.
\end{proof}
Here is the main theorem of this section.
\begin{theorem}\label{goodp3theorem}
For all $k,r\in \mathbb{N}$, there exists $\upsilon(k,r)\in \mathbb{N}$ with the following property. Let $(G,L)$ be an instance of the \textsc{List-$k$-Coloring Problem} where $G$ is $rP_3$-free. Then there exists a spanning $(G,L)$-profile $\Upsilon(G,L)$ with the following specifications.
	\begin{itemize}
		\item \sloppy $|\Upsilon(G,L)|\leq \mathcal{O}(|V(G)|^{\upsilon(k,r)})$, and $\Upsilon(G,L)$ can be computed from $(G,L)$ in time $ \mathcal{O}(|V(G)|^{\upsilon(k,r)})$.
		\item For every $(G,L')\in \Upsilon(G,L)$, $G$ has no $L'$-good $P_3$.
	\item If $G$ admits a frugal $L$-coloring, then for some $(G,L')\in \Upsilon(G,L)$, $G$ admits a frugal $L'$-coloring. 
	\end{itemize}
\end{theorem} 
\begin{proof}
	If $k=1$, then by setting $\upsilon(1,r)=1$ and $\Upsilon(G,L)=\{(G,L)\}$, we are done. So we may assume that $k\geq 2$. Let $(\gamma_i=(I_1^i, I_2^i,I_3^i):i\in [m])$ be an enumeration of all good triples of subsets of $[k]$, such that for $i,j\in [m]$, $i>j$ implies $|I^i_1|+|I^i_2|+|I^i_3|\leq |I^j_1|+|I^j_2|+|I^j_3|$. Note that $m\leq 2^{3k}$, and so this enumeration can be computed in constant time.
	
	\sta{\label{f<r}For every list assignment $K$ of $G$ and every $i\in [m]$, we have $f_{K,\gamma_i}(G)\leq r-1$.}
For otherwise $G$ contains $r$ mutually anticomplete induced $P_3$'s, which violates the assumption of Theorem \ref{goodp3theorem} that $G$ is $rP_3$-free. This proves \eqref{f<r}.

\sta{\label{goodthmseq}There exists a sequence $(\Lambda_0,\ldots, \Lambda_m)$ of spanning $(G,L)$-profiles, where $\Lambda_0=\{(G,L)\}$, and for each $i\in [m]$, the following hold.
		\begin{itemize}
		\item $|\Lambda_i|\leq \mathcal{O}(|V(G)|^{i(r-1)(3k-3)})$ and  $\Lambda_i$ can be computed from $\Lambda_{i-1}$ in time $ \mathcal{O}(|V(G)|^{i(r-1)(3k-2)})$.
		\item For every $(G,L')\in \Lambda_i$, $G$ has no induced $P_3$ of $L'$-type in $\{\gamma_j:j\in [i]\}$.
	\item If $G$ admits a frugal $L''$-coloring for some $(G,L'')\in \Lambda_{i-1}$, then for some $(G,L')\in \Lambda_{i}$, $G$ admits a frugal $L'$-coloring. 
\end{itemize}}

\sloppy We generate this sequence recursively. To initiate, note that $G$ has no $L$-good $P_3$ of $L$-weight larger that $|I^1_1|+|I^1_2|+|I^1_3|$. Thus, we may apply Lemma \ref{goodp3lem2} to $(G,L)$ and $\gamma_1$, obtaining a spanning $(G,L)$-profile $\Upsilon_2(G,L)$ which satisfies Lemma \ref{goodp3lem2}. As result, defining $\Lambda_1=\Upsilon_2(G,L)$, then by \eqref{f<r}, $\Lambda_1$ satisfies the bullet conditions of \eqref{goodthmseq} for $i=1$. Next, assume that for some $i\in \{2,\ldots, m\}$, the $(G,L)$-profile $\Lambda_{i-1}$, satifying  the bullet conditions of \eqref{goodthmseq}, is computed. In particular, for every $(G,L'')\in \Lambda_{i-1}$, $G$ has no induced $P_3$ of $L''$-type in $\{\gamma_j:j\in [i-1]\}$. As a result, $G$ has no $L''$-good $P_3$ of $L''$-weight larger that $|I^i_1|+|I^i_2|+|I^i_3|$. Thus, we may apply Lemma \ref{goodp3lem2} to $(G,L'')$ and $\gamma_i$, obtaining a spanning $(G,L)$-profile $\Upsilon_2(G,L'')$ which satisfies Lemma \ref{goodp3lem2}. Let $\Lambda_i=\bigcup_{(G,L'')\in \Lambda_{i-1}}\Upsilon_2(G,L'')$. We claim that $\Lambda_i$ satisfies the bullet conditions of \eqref{goodthmseq}. To see this, from \eqref{f<r} and the first bullet of Lemma \ref{goodp3lem2}, we deduce that for every $(G,L'')\in \Lambda_{i-1}$,  $|\Upsilon_2(G,L'')|\leq \mathcal{O}(|V(G)|^{f_{L'',\gamma_i}(G)(3k-3)})=\mathcal{O}(|V(G)|^{(r-1)(3k-3)})$ and $\Upsilon_2(G,L'')$ can be computed from $(G,L'')$ in time $\mathcal{O}(|V(G)|^{f_{L'',\gamma_i}(G)(3k-2)})=\mathcal{O}(|V(G)|^{(r-1)(3k-2)})$. This, along with the fact that $\Lambda_{i-1}$ satisfies the first bullet of \eqref{goodthmseq}, implies that $|\Lambda_i|\leq \mathcal{O}(|V(G)|^{(i-1)(r-1)(3k-3)})\mathcal{O}(|V(G)|^{(r-1)(3k-3)})=\mathcal{O}(|V(G)|^{i(r-1)(3k-3)})$, and $\Lambda_i$ can be computed from $\Lambda_{i-1}$ in time $\mathcal{O}(|V(G)|^{(i-1)(r-1)(3k-3)})\mathcal{O}(|V(G)|^{(r-1)(3k-2)})=\mathcal{O}(|V(G)|^{i(r-1)(3k-2)})$. Therefore, $\Lambda_i$ satisfies the first bullet of \eqref{goodthmseq}.
Moreover, for every $(G,L')\in \Lambda_i$, say $(G,L')\in \Upsilon_2(G,L'')$ for some $(G,L'')\in \Lambda_{i-1}$, by second bullet of Lemma \ref{goodp3lem2},  $G$ has no induced $P_3$ of $L'$-type $\gamma_i$. Also, since $(G,L')$ is a $(G,L'')$-refinement, by the second bullet \eqref{goodthmseq} for $\Lambda_{i-1}$, $G$ has no induced $P_3$ of $L'$-type in $\{(I_1^j, I_2^j,I_3^j):j\in [i-1]\}$. It follows that $G$ has no induced $P_3$ of $L'$-type in $\{(I_1^j, I_2^j,I_3^j):j\in [i]\}$. So $\Lambda_i$ satisfies the second bullet of \eqref{goodthmseq}. Finally, the third bullet of Lemma \ref{goodp3lem2} implies that, if $G$ admits a frugal $L''$-coloring for some $(G,L'')\in \Lambda_{i-1}$, then for some $(G,L')\in \Lambda _i$, $G$ admits a frugal $L'$-coloring. So $\Lambda_i$ satisfies the third bullet of  \eqref{goodthmseq}. This proves \eqref{goodthmseq}.\vsp

Now, let $(\Lambda_1,\ldots, \Lambda_m)$ be as in \eqref{goodthmseq}. Let $\Upsilon(G,L)=\Lambda_m$. Then, since $m\leq 2^{3k}$, by the first bullet of \eqref{goodthmseq} for $i=m$, we have $|\Upsilon(G,L)|\leq \mathcal{O}(|V(G)|^{(r-1)(3k-3)2^{3k}})$, and by the first bullet of \eqref{goodthmseq} for $i=0,1,\ldots, m$, $\Upsilon(G,L)$ can be computed  in time $\mathcal{O}(|V(G)|^{(r-1)(3k-2)2^{3k}})$. So by setting $\upsilon(k,r)=(r-1)(3k-2)2^{3k}$, $\Upsilon(G,L)$ satisfies the first bullet of Theorem \ref{goodp3theorem}. Also, by the second bullet of \eqref{goodthmseq} for $i=m$, for every $(G,L')\in \Upsilon(G,L)$, $G$ has no induced $P_3$ of $L'$-type in $\{\gamma_i:i\in [m]\}$, and so $G$ has no $L'$-good $P_3$. Therefore, $\Upsilon(G,L)$ satisfies the second bullet of Theorem \ref{goodp3theorem}. Finally, applying the third bullet of \eqref{goodthmseq} to $i=0,1,\ldots,m$ consecutively, it follows that if $G$ admits a frugal $L$-coloring, then for some $(G,L')\in \Upsilon(G,L)$, $G$ admits a frugal $L'$-coloring. Hence, $\Upsilon(G,L)$ satisfies the third bullet of Theorem \ref{goodp3theorem}. This completes the proof.
\end{proof}

\section{Five colors and vertices with large lists}\label{23sec}
		
In this section, we take the last major step towards the proof of  Theorem \ref{thm:main}: we show that essentially every  instance of the \textsc{List-$5$-Coloring Problem} which has at least one vertex of list-size three or more and no good $P_3$'s can be reduced in polynomial time to a ``smaller'' instance. We prove this formally in Theorem \ref{23}, whose proof relies crucially on two lemmas, and in order to state them, we need another definition. Let $k\in \mathbb{N}$ and $(G,L)$ be an instance of the \textsc{List-$k$-Coloring Problem}.  We denote by $G^L$ the graph with $V(G^L)=V(G)$ and $E(G^L)=\{uv\in E(G):L(u)\cap L(v)\neq \emptyset\}$.  Note that $(G,L)$ and $(G^L,L)$ have the same state of feasibility. But $G^L$ is not necessarily an induced subgraph of $G$, and so for our purposes, it seems dangerous to consider $(G^L,L)$ as a `simplified' instance to investigate. However, it turns out that we may still take advantage of certain properties of $G^L$. For example, the following lemma proposes a useful interaction between frugality and good $P_3$'s in terms of vertex degrees in $G^L$.
\begin{lemma}\label{smallgooddeg}
	Let $(G,L)$ be an instance of the \textsc{List-$k$-Coloring Problem} such that $|L(v)|\neq 1$ for every $v\in V(G)$, and $G$ has no $L$-good $P_3$. If $G$ admits a frugal $L$-coloring  $\phi$, then for every vertex $v\in V(G)$, $\phi$ assigns mutually distinct colors to all vertices in $N_{G^L}[v]$, and in particular, we have $|N_{G^L}(v)|<k$.
\end{lemma}
\begin{proof}
	Suppose not. Then since $\phi$ is proper, there exist two vertices $u,w\in N_{G^L}(v)$ such that $u$ and $w$ are nonadjacent  in $G$ and $\phi(u)=\phi(w)\in L(u)\cap L(w)$. Also, since  $u,w\in N_{G^L}(v)$, both $L(u)\cap L(v)$ and $L(v)\cap L(w)$ are nonempty.  But then $u-v-w$ is an $L$-good $P_3$ in $G$, a contradiction. This completes the proof.
\end{proof}

 The following technical lemma also unravels the structural properties of the second neighborhood of certain vertices in $G^L$. We use this lemma  extensively while proving  Theorem \ref{23}.
		
\begin{lemma} \label{ABhomo}
Let  $(G,L)$ be an instance of the \textsc{List-$5$-Coloring Problem} such that $|L(u)|\in \{0,2,3\}$ for all $u\in V(G)$ and $G$ has no $L$-good $P_3$. Moreover, suppose that there exists a vertex $u_0\in V(G)$ with $L(u_0)=\{1,2,3\}$. In addition, let $A=N_{G^L}(u_0)\cap L^{(4)}$,  $B=N_{G^L}(u_0)\cap L^{(5)}$, $A'=\{w\in N^2_{G^L}(u_0):  N_{G^L}(w)\cap A\neq\emptyset\}$, $B'=\{w\in N^2_{G^L}(u_0): N_{G^L}(w)\cap B\neq\emptyset\}$. Then the following hold.
	\begin{itemize}
		\item $|L(u)|\in \{2,3\}$ for every $u\in N^2_{G^L}[u_0]$.
		\item $L(w)=\{4,5\}$ for every $w\in N^2_{G^L}(u_0)$.
		\item $N^2_{G^L}(u_0)=A'\cup B'$.
		\item Both $A$ and $B$ are cliques of $G^L$.
		\item For every vertex $w\in N^2_{G^L}(u_0)$, $w$ is either complete or anticomplete to $A$ in $G^L$, and either complete or anticomplete to $B$ in $G^L$. In particular, $A'$ is complete to $A$ in $G^L$, and $B'$ is complete to $B$ in $G^L$.
		\item Both $A'$ and $B'$ are cliques of $G^L$.
		\item If in addition, $|N_{G^L}^2(u_0)|\geq 2$ and $|A'|,|B'|\leq 1$, then 
		\begin{itemize}
\item[-] $|A'|=|B'|=1$ and $A'\cap B'=\emptyset$;
\item[-] $A,B\neq \emptyset$ and $A\cap B=\emptyset$; 
\item[-] $A'$ is anticomplete to $B$ in $G$ and $B'$ is anticomplete to $A$ in $G$, and;
\item[-] for every $a\in A$ and every $b\in B$, we have $L(a)\cap L(b)=\emptyset$.
		\end{itemize} 
	\end{itemize}	
	\end{lemma}
\begin{proof}
	The first bullet follows directly from $L(u_0)=\{1,2,3\}\neq \emptyset$, $u_0\in N^2_{G^L}[u_0]$ and $G^L|N^2_{G^L}[u_0]$ being connected.
	
	To see the second bullet, let $w\in N^2_{G^L}(u_0)$ and $v\in N_{G^L}(u_0)\cap N_{G^L}(w)$. If $L(u_0)\cap L(w)\neq \emptyset$, then $u_0w\notin E(G)$, and so from $u_0v,vw\in E(G^L)$ and first bullet of  Lemma \ref{ABhomo}, it follows that $u_0-v-w$ is an $L$-good $P_3$ in $G$, which is impossible. Consequently, we have $L(w)\subseteq [5]\setminus L(u_0)=\{4,5\}$, and so by the first bullet of Lemma \ref{ABhomo}, we have $L(w)=\{4,5\}$. This proves the second bullet of  Lemma \ref{ABhomo}.
	
	 To verify the third bullet, note that the inclusion $A'\cup B'\subseteq N^2_{G^L}(u_0)$ is clear. Now, let $w\in N^2_{G^L}(u_0)$ and $v\in N_{G^L}(u_0)\cap N_{G^L}(w)$. Then by the second bullet of  Lemma \ref{ABhomo}, we have $L(v)\cap \{4,5\}=L(v)\cap L(w)\neq \emptyset$, and so $v\in A\cup B$, which in turn implies that $w\in A'\cup B'$. So $N^2_{G^L}(u_0)\subseteq A'\cup B'$, and the third bullet of  Lemma \ref{ABhomo} follows.
	
	To see the fourth bullet, suppose for a contradiction that there exist $a_1,a_2\in A$ with $a_1a_2\notin E(G^L)$. Therefore, since $4\in L(a_1)\cap L(a_2)$, we have $a_1a_2\notin E(G)$. This, together with $a_1,a_2\in N_{G^L}(u_0)\subseteq N_G(u_0)$ and the first bullet of  Lemma \ref{ABhomo}, implies that $a_1-u_0-a_2$ is an $L$-good $P_3$ in $G$, a contradiction. So $A$ is a clique of $G^L$. Similarly, one can show that $B$ is also a clique of $G^L$,  and so the fourth bullet of  Lemma \ref{ABhomo} follows.
	
	Now we argue the fifth bullet. Suppose for a contradiction that there exists $w\in N^2_{G^L}(u_0)$, such that in $G^L$, $w$ has both a neighbor $x$ and a non-neighbor $y$ in either $A$ or $B$, say the former. By the fourth bullet of Lemma \ref{ABhomo}, we have $xy\in E(G^L)\subseteq E(G)$. Also, by the second bullet of Lemma \ref{ABhomo}, we have $4\in L(w)\cap L(x)\cap L(y)$, which in turn implies that $wy\notin E(G)$; that is, $w-x-y$ is an induced $P_3$ in $G$. From this and the second bullet of Lemma \ref{ABhomo}, it follows that $w-x-y$ is an $L$-good $P_3$ in $G$, a contradiction. The case $x,y\in B$ can be handled similarly. This proves the fifth bullet of Lemma \ref{ABhomo}.
	
For the sixth bullet, suppose for a contradiction that  $a'_1,a'_2\in A'$ are not adjacent in $G^L$. Then by the second bullet of  Lemma \ref{ABhomo}, we have $L(a'_1)=L(a'_2)=\{4,5\}$, and so $a'_1$ and $a'_2$ are not adjacent in $G$. Note that since $A'\neq \emptyset$, by the definition of $A'$, we have $A\neq \emptyset$, and so we may pick a vertex $a_0\in A$. Therefore, by the sixth bullet of Lemma \ref{ABhomo}, $a'_1-a_0-a'_2$ is an induced $P_3$ in $G$ with $4\in L(a'_1)\cap L(a_0)\cap L(a'_2)$, which along with the first bullet of Lemma \ref{ABhomo}, implies that $a'_1-a_0-a'_2$ is an $L$-good $P_3$ in $G$, a contradiction. So the sixth bullet of Lemma \ref{ABhomo} follows.
	
The rest of the proof aims to verify the seventh bullet of Lemma \ref{ABhomo}. For the first dash, from $|N_{G^L}^2(u_0)|\geq 2$ and the third bullet of Lemma \ref{ABhomo}, we have $|A'\cup B'|\geq 2$, which along with $|A'|,|B'|\leq 1$, implies that $|A'|=|B'|=1$ and $A'\cap B'=\emptyset$, as desired. Henceforth, we assume $A'=\{a'\}$ and $B'=\{b'\}$ for distinct $a',b'$. Note that by the second bullet of Lemma \ref{ABhomo}, we have $L(a')=L(b')=\{4,5\}$.

For the second dash, note that the fact that $|A'|=|B'|=1$ along with the definition of $A'$ and $B'$, implies that $A,B\neq \emptyset$. Next we show that $A\cap B=\emptyset$. Suppose not. Let $z\in A\cap B$. By the fifth bullet of Lemma \ref{ABhomo}, $a'$ and $b'$ are adjacent to $z$ in $G^L$. Thus, from $z\in A\cap B$ and again the fifth bullet of Lemma \ref{ABhomo}, we deduce that $A'\cup B'=\{a',b'\}$ is complete to $A\cup B$ in $G^L$. But then from the definition of $A'$ and $B'$, it follows that $A=B$, which in turn implies that $A'=B'$, a contradiction with the first dash. So $A\cap B=\emptyset$, as desired. 

To see the third dash, suppose for a contradiction that $a'$ has a neighbor $b\in B$ in $G$. Then, since $5\in L(a')\cap L(b)$, we have $a'b\in E(G^L)$. But then from the definition of $B'$, we have $a'\in B'$, and so $a'\in A'\cap B'$, which violates the first dash. Note that if $b'$ has a neighbor in $A$, then a contradiction can be derived similarly.

Finally, we prove the fourth dash. Suppose not. Let $L(a)\cap L(b)\neq \emptyset$ for distinct $a\in A$ an $b\in B$. If $ab\notin E(G)$, then from $a,b\in N_{G^L}(u_0)$ and the first bullet of Lemma \ref{ABhomo}, we deduce that $a-u_0-b$ is an $L$-good $P_3$, which is impossible. As a result, we have $ab\in E(G)$, and so $ab\in E(G^L)$. By the fifth bullet of Lemma \ref{ABhomo}, $a'$ is adjacent to $a$ in $G^L$, and by the third dash, $a'$ is not adjacent to $b$ in $G$. So $a'-a-b$ is an induced $P_3$ in $G$. This, along with $4\in L(a')\cap L(a)$, $5\in L(a')\cap L(b)$ and $L(a)\cap L(b)\neq \emptyset$ implies that $a'-a-b$ is an $L$-good $P_3$ in $G$, a contradiction. This proves the fourth dash of the seventh bullet of Lemma \ref{ABhomo}, and so concludes the proof.
\end{proof} 	
Let $k\in \mathbb{N}$ and $(G,L)$ be an instance of the \textsc{List-$k$-Coloring Problem}. We define $p(G,L)=|V(G)|+\sum_{v\in V(G)}|L(v)|$. It is immediate from the definition that $p(G,L)\leq (k+1)|V(G)|$. For a $(G,L)$-refinement $(G'L')$, we say that $(G',L')$ \textit{represents} $(G,L)$ if the following hold.
	\begin{enumerate}[(R1)]
		\item\label{r1} $p(G',L')< p(G,L)$.
		\item\label{r2} If $G$ admits a frugal $L$-coloring, then $G'$ admits a frugal $L'$-coloring. 
		\item\label{r3} If $G'$ admits an $L'$-coloring, then $G$ admits an  $L$-coloring.
\end{enumerate}
		\begin{theorem}\label{23}
		Let  $(G,L)$ be an instance of the \textsc{List-$5$-Coloring Problem} such that $|L(v)|\neq 1$ for all $v\in V(G)$ and $G$ has no $L$-good $P_3$. Moreover, suppose that there exists a vertex $u_0\in V(G)$ with $|L(u_0)|\geq 3$. Then there exists a $(G,L)$-refinement $(\tilde{G},\tilde{L})$ with the following specifications. 
		 	\begin{itemize}
		 	\item $(\tilde{G},\tilde{L})$ can be computed from $(G,L)$ in time $\mathcal{O}(|V(G)|^2)$.
		 		\item $|\tilde{L}(v)|\neq 1$ for all $v\in V(\tilde{G})$.
		 	\item $(\tilde{G},\tilde{L})$ represents $(G,L)$.
		 \end{itemize}
		\end{theorem}
	\begin{proof}
\sloppy Without loss of generality, we may assume that $\{1,2,3\}\subseteq L(u_0)$. We define the four sets  $A=N_{G^L}(u_0)\cap L^{(4)}$,  $B=N_{G^L}(u_0)\cap L^{(5)}$, $A'=\{w\in N^2_{G^L}(u_0): N_{G^L}(w)\cap A\neq \emptyset\}$ and $B'=\{w\in N^2_{G^L}(u_0): N_{G^L}(w)\cap B\neq \emptyset\}$ as in Lemma \ref{ABhomo}. For every vertex $u\in V(G)$, let $\Phi_u$ be the set of all frugal $L|_{N^2_{G^L}[u]}$-colorings of $G|N^2_{G^L}[u]$.  Consider the following algorithm, called \textit{algorithm} \textbf{\textsf{A}}, which, given $G$, $L$ and $u_0$, computes a $(G,L)$-refinement $(G^*,L^*)$.
	\begin{enumerate}[\textit{Step }1:]
		\item \label{s1} Using BFS, compute $N_{G^L}(v)$ and $N^2_{G^L}(v)$ for every $v\in V(G)$. Go to step \ref{s2}, and from each step, proceed to the one below unless instructed otherwise.
		\item \label{s2} Compute $A$, $B$, $A'$ and $B'$.
		\item \label{s3} If $|N_{G^L}(u)|\geq 5$ for some $u\in V(G)$, then compute $G^*=G$,  $L^*(v)=\emptyset$ for every $v\in V(G^*)$. Return $(G^*,L^*)$.
		\item \label{s4} If $|N_{G^L}(u)|< |L(u)|$ for some $u\in V(G)$, then compute $G^*=G-u$, $L^*(v)=L(v)$ for every $v\in V(G^*)$. Return $(G^*,L^*)$.
		\item \label{s5} If $|N^2_{G^L}(u)|\leq 1$ for some $u\in V(G)$, then
		\begin{enumerate}
			\item \label{s5a} Compute $\Phi_u$ by brute-forcing.
			\item \label{s5b} If $\Phi_u=\emptyset$, then compute $G^*=G$ and  $L^*(v)=\emptyset$ for every $v\in V(G^*)$. Return $(G^*,L^*)$.
				\item \label{s5c} Otherwise, Compute $G^*=G-N_{G^L}[u]$, $L^*(v)=\{i\in L(v): \phi(v)=i\text{ for some }\phi\in \Phi_u\}$ for every $v\in N^2_{G^L}(u)$ and $L^*(v)=L(v)$ for every $v\in V(G^*)\setminus N^2_{G^L}(u)$. Return $(G^*,L^*)$.
		\end{enumerate}
	
	\item \label{s6} If $|A'|\geq 2$, then compute $G^*=G$, $L^*(a)=L(a)\setminus \{4,5\}$ for every $a\in A$ and $L^*(v)=L(v)$ for every $v\in V(G^*)\setminus A$. Return $(G^*,L^*)$.
	\item \label{s7} If $|B'|\geq 2$, then compute $G^*=G$, $L^*(b)=L(b)\setminus \{4,5\}$ for every $b\in B$ and $L^*(v)=L(v)$ for every $v\in V(G^*)\setminus B$.  Return $(G^*,L^*)$.
	\item \label{s8} If there exist distinct vertices $a_1,a_2\in A$ with $L(a_1)=L(a_2)$ and $|L(a_1)|=|L(a_2)|=2$, then compute $G^*=G$, $L^*(a')=L(a')\setminus \{4\}$ for every $a'\in A'$ and $L^*(v)=L(v)$ for every $v\in V(G^*)\setminus A'$. Return $(G^*,L^*)$.
	\item \label{s9} If there exist distinct vertices $b_1,b_2\in B$ with $L(b_1)=L(b_2)$ and $|L(b_1)|=|L(b_2)|=2$, then compute $G^*=G$, $L^*(b')=L(b')\setminus \{5\}$ for every $b'\in B'$ and $L^*(v)=L(v)$ for every $v\in V(G^*)\setminus B'$. Return $(G^*,L^*)$.
\item \label{s10} If $|N_{G^L}(u_0)|\geq 4$, then compute $G^*=G$, $L^*(a')=L(a')\setminus \{4\}$ for every $a'\in A'$, $L^*(b')=L(b')\setminus \{5\}$ for every $b'\in B'$ and $L^*(v)=L(v)$ for every $v\in V(G^*)\setminus (A'\cup B')$. Return $(G^*,L^*)$.
\item \label{s11} Compute a minimal subset $M$ of $L(u_0)$ such that $M\cap L(a)\neq \emptyset$ for some $a\in A$ and $M\cap L(b)\neq \emptyset$ for some $b\in B$. Choose $i\in M\cap L(a)$ and $j\in M\cap L(b)$. Compute $G^*=G|(\{u_0,a,b\}\cup (V(G)\setminus N_{G^L}[u_0]))$, $L^*(u_0)=M$, $L(a)=\{i,4\}$, $L(b)=\{j,5\}$ and $L^*(v)=L(v)$ for every $v\in V(G^*)\setminus \{u_0,a,b\}$. Return $(G^*,L^*)$.
	\end{enumerate}
	As a  general property of algorithm \textbf{\textsf{A}}, note that for each step other than steps \ref{s1}, \ref{s2} and \ref{s5a}, if the corresponding `if condition' is satisfied, then the algorithm terminates at that step. In particular,

\sta{\label{size4done} Suppose that $|L(u)|\geq 4$ for some $u\in V(G)$. Then algorithm \textbf{\textsf{A}} terminates at or before step \ref{s5}.}
Let $i\in [5] $ with $[5]\setminus L(u)\subseteq \{i\}$. We claim that  $N^2_{G^L}(u)=\emptyset$. Suppose not. Let $w\in N^2_{G^L}(u)$. Then since $G$ has no good $P_3$, we have $L(u)\cap L(w)=\emptyset$, and so $L(w)\subseteq \{i\}$. Also, there exists $v\in N_{G^L}(u)$ which is adjacent to $w$ in $G^L$, and so $L(v)\cap L(w)\neq\emptyset$. Thus, $L(w)=\{i\}$, and so $|L(w)|=1$, a contradiction. This proves the claim. But then the `if condition' in step \ref{s5} is satisfied, and so algorithm \textbf{\textsf{A}} terminates at or before step \ref{s5}. This proves \eqref{size4done}.\vsp

We deduce:

\sta{\label{getbullets} If algorithm \textbf{\textsf{A}} does not stop at steps \ref{s3}-\ref{s7}, then all seven bullets of Lemma \ref{ABhomo} hold.}

Note that algorithm \textbf{\textsf{A}} does not terminate at step \ref{s5}. This has two consequences. First, by \eqref{size4done} and the assumption of Theorem \ref{23}, we have $|L(u)|\in \{0,2,3\}$ for every $u\in V(G)$. In particular, we have $L(u_0)=\{1,2,3\}$. Second, the `if condition' of step \ref{s5} is not satisfied, and in particular $|N_{G^L}^2(u_0)|\geq 2$.  In addition, since algorithm \textbf{\textsf{A}} does not stop at steps \ref{s6} and \ref{s7}, the `if condition' in these two steps is not satisfied, and so $|A'|,|B'|\leq 1$. Therefore, all seven bullets of Lemma \ref{ABhomo} hold. This proves \eqref{getbullets}.

\sta{\label{23runtime} Algorithm \textbf{\textsf{A}} terminates in finite time. Indeed, it runs in time $\mathcal{O}(|V(G)|^2)$.}

Note that if the algorithm does not terminate in steps \ref{s3}-\ref{s10}, then it arrives at step \ref{s11}, and by \eqref{getbullets}, all seven bullets of Lemma \ref{ABhomo} hold.  In particular, by the second dash of the seventh bullet of Lemma \ref{ABhomo}, we have $A,B\neq \emptyset$. As a result, the set $M$ mentioned in step \ref{s11} is well-defined. So algorithm \textbf{\textsf{A}} executes step \ref{s11}, and stops in finite time. We leave the reader to check the straightforward fact that the overall running time of algorithm \textbf{\textsf{A}} is $\mathcal{O}(|V(G)|^2)$.
This proves \eqref{23runtime}.\vsp

We need to show that the output $(G^*,L^*)$ of algorithm \textbf{\textsf{A}} represents $(G,L)$. The proof is broken into several statements, below.

\sta{\label{geq5done} If algorithm \textbf{\textsf{A}} terminates at step \ref{s3}, then $(G^*,L^*)$ represents $(G,L)$.}
Note that $|L^*(u_0)|=0<3\leq |L(u_0)|$, and so $(G^*,L^*)$ satisfies (R\ref{r1}). Also, the `if condition' of step \ref{s3} is satisfied, and there exists a vertex $u\in V(G)$ with $|N_{G^L}(u)|\geq 5$. Thus, $L$ being a $5$-list assignment, by Lemma \ref{smallgooddeg}, $G$ admits no frugal $L$-coloring, and so $(G^*,L^*)$ vacuously satisfies (R\ref{r2}). Moreover, since $L^*(v)=\emptyset$ for every $v\in V(G^*)$, $G^*$ admits no $L^*$-coloring, and so $(G^*,L^*)$ vacuously satisfies (R\ref{r3}), as well. This proves \eqref{geq5done}.

\sta{\label{smalldegdone} If algorithm \textbf{\textsf{A}} terminates at step \ref{s4}, then $(G^*,L^*)$ represents $(G,L)$.}
The `if condition' of step \ref{s4} is satisfied, and so we have $|N_{G^L}(u)|< |L(u)|$ for some $u\in V(G)$, $G^*=G-u$ and $L^*(v)=L(v)$ for every $v\in V(G^*)$. As a result, $|V(G^*)|<|V(G)|$, and so $(G^*,L^*)$ satisfies (R\ref{r1}). Moreover, if $G$ admits a frugal $L$-coloring $\phi$, then by Lemma \ref{frugalspread}, $\phi|_{V(G^*)}$ is a frugal $L^*$-coloring of $G^*$, and so $(G^*,L^*)$ satisfies (R\ref{r2}). Now, suppose that $G^*$ admits an $L^*$-coloring $\psi$. Since $|\{\psi(v):v\in N_{G^L}(u)\}|\leq |N_{G^L}(u)|<|L(u)|$, there exists a color $j\in L(u)\setminus \{\psi(v):v\in N_{G^L}(u)\}$. Hence, extending $\phi$ to $G$ by defining $\psi(u)=j$, we obtain an $L$-coloring of $G$, and so $(G^*,L^*)$ satisfies (R\ref{r3}). This proves \eqref{smalldegdone}.

\sta{\label{C<1done2} If algorithm \textbf{\textsf{A}} terminates at step \ref{s5}, then $(G^*,L^*)$ represents $(G,L)$.}

The `if condition' of step \ref{s5} is satisfied, and so we have $|N^2_{G^L}(u)|\leq 1$ for some $u\in V(G)$. Now, suppose that $\Phi_u=\emptyset$. Then algorithm \textbf{\textsf{A}} terminates at step \ref{s5b}, $G^*=G$ and $L^*(v)=\emptyset$ for every $v\in V(G)$. Note that $|L^*(u_0)|=0<3\leq |L(u_0)|$, and so $(G^*,L^*)$ satisfies (R\ref{r1}). Also, from $\Phi_u=\emptyset$, it follows that $G|N^2_{G^L}[u]$ admits no frugal $L|_{N^2_{G^L}[u]}$-coloring, and so $G$ admits no frugal $L$-coloring. Thus, $(G^*,L^*)$ vacuously satisfies (R\ref{r2}). Moreover, since $L^*(v)=\emptyset$ for every $v\in V(G^*)$, $G^*$ admits no $L^*$-coloring, and $(G^*,L^*)$ vacuously satisfies (R\ref{r3}), as well.

Therefore, we may assume that $\Phi_u\neq \emptyset$. Then algorithm \textbf{\textsf{A}} terminates at step \ref{s5c}, $~{G^*= G- N_{G^L}[u]}$, $L^*(v)=\{i\in L(v): \phi(v)=i\text{ for some }\phi\in \Phi_u\}$ for every $v\in N^2_{G^L}(u)$ and $L^*(v)=L(v)$ for every $v\in V(G^*)\setminus N^2_{G^L}(u)$. It follows that $|V(G^*)|< |V(G)|$, and so $(G^*,L^*)$ satisfies (R\ref{r1}). For (R\ref{r2}), suppose that $G$ admits a frugal $L$-coloring $\psi$. Then $\phi=\psi|_{N^2_{G^L}[u]}$ is readily seen to be a frugal $L|_{N^2_{G^L}[u]}$-coloring of $G|N^2_{G^L}[u]$; that is, $\phi\in \Phi_u$. So for every $v\in N_{G^L}^2(u)$, $\psi(v)=\phi(v)\in L^*(v)$. Also $\psi(v)\in L(v)=L^*(v)$ for all $v\in V(G^*)\setminus N_{G^L}^2(u)$. Thus, $\psi$ being a frugal $L$-coloring of $G$,  it follows from Lemma \ref{frugalspread} that $\psi|_{V(G^*)}$ is a frugal $L^*$-coloring of $G^*$. This shows that $(G^*,L^*)$ satisfies (R\ref{r2}).
Now we need to argue that $(G^*,L^*)$ satisfies (R\ref{r3}). Suppose that $G^*$ admits an $L^*$-coloring $\psi'$. Since $|N^2_{G^L}(u)|\leq 1$ and $\psi'(c)\in L^*(c)$ for every $c\in N^2_{G^L}(u)$, there exists $\phi'\in \Phi_u$ such that $\phi'(c)=\psi'(c)$ for all $c\in N^2_{G^L}(u)$. We define a coloring $\theta:V(G)\rightarrow [5]$ as follows. For every $v\in N_{G^L}[u_0]$, let $\theta(v)=\phi'(v)$. For all $c\in N^2_{G^L}(u)$, let $\theta(c)=\phi'(c)=\psi'(c)$. For every $v\in V(G)\setminus N_{G^L}^2[u_0]=V(G^*)\setminus N_{G^L}^2(u_0)$, let $\theta(v)=\psi'(v)$. We claim that $\theta$ is an $L$-coloring of $G$. To see this, note that since $\phi'\in \Phi_u$, we have $\theta(v)=\phi'(v)\in L(v)$ for every $v\in N^2_{G^L}[u_0]$. Also,  $\theta(v)=\psi'(v)\in L^*(v)=L(v)$ for every $v\in  V(G)\setminus N_{G^L}^2[u_0]$. So $\theta(v)\in L(v)$ for all $v\in V(G)$. It remains to show that $\theta$ is proper. Let $xy\in E(G)$. If $x,y\in  N^2_{G^L}[u_0]$, then since $\phi'$ is proper, we have $\theta(x)=\phi'(x)\neq \phi'(y)=\theta(y)$. If $x,y\in V(G)\setminus N_{G^L}[u_0]=V(G^*)$, then since $\psi'$ is proper, $\theta(x)=\psi'(x)\neq \psi'(y)=\theta(y)$. Finally, let $x\in N_{G^L}[u_0]$ and $y\in V(G)\setminus N_{G^L}^2[u_0]$. Then $xy\notin E(G^L)$, and so $L(x)\cap L(y)=\emptyset$. Hence $\theta(x)\neq\theta(y)$, and $\theta$ is proper. This proves \eqref{C<1done2}.

\sta{\label{A'2} If algorithm \textbf{\textsf{A}} terminates at step \ref{s6}, then $(G^*,L^*)$ represents $(G,L)$.}
The `if condition' of step \ref{s6} is satisfied, and so we have $|A'|\geq 2$, $G^*=G$, $L^*(a)=L(a)\setminus \{4,5\}$ for every $a\in A$ and $L^*(v)=L(v)$ for every $v\in V(G^*)\setminus A$.
Note that from $|A'|\geq 2$ and the definition of $A'$, it follows that $A\neq \emptyset$, and so for every $a\in A$, we have $4\in L(a)\setminus L^*(a)$, which in turn implies that $|L^*(a)|<|L(a)|$. This shows that $(G^*,L^*)$ satisfies (R\ref{r1}).

For (R\ref{r2}), suppose that $G$ admits a frugal $L$-coloring $\phi$. Let $a'_1,a'_2\in A'$ be distinct. Since the algorithm does not stop at step \ref{s5}, by \eqref{size4done} and the assumption of Theorem \ref{23}, we have $|L(u)|\in \{0,2,3\}$ for every $u\in V(G)$. In particular, we have $L(u_0)=\{1,2,3\}$. So by the second bullet of Lemma \ref{ABhomo}, we have $L(a'_1)=L(a'_2)=\{4,5\}$, by the fifth bullet of Lemma \ref{ABhomo}, $a'_1$ and $a'_2$ are complete to $A$ in $G^L$ (and so in $G$), and by the sixth bullet of Lemma \ref{ABhomo}, $a'_1$ and $a'_2$ are adjacent in $G^L$ (and so in $G$). As a result, we have $\{\phi(a'_1),\phi(a'_2)\}=\{4,5\}$, and for every $a\in A$, we have $\phi(a)\in L(v)\setminus \{\phi(a'_1),\phi(a'_2)\}=L(v)\setminus\{4,5\}=L^*(a)$. In addition, we have $\phi(v)\in L(v)=L^*(v)$ for every $v\in V(G^*)\setminus A$. Thus, $\phi$ being a frugal $L$-coloring $\phi$, it follows from Lemma \ref{frugalspread} that $\phi|_{V(G^*)}$ is an $L^*$-coloring of $G^*$, and so $(G^*,L^*)$ satisfies (R\ref{r2}). Finally, note that $(G^*,L^*)$ is a spanning $(G,L)$-refinement, and by Lemma \ref{spanspread}, $(G^*,L^*)$ satisfies (R\ref{r3}). This proves \eqref{A'2}.\vsp

The reader may have noticed that steps \ref{s6} and \ref{s7} of algorithm \textbf{\textsf{A}} are symmetrical with respect to $A$ and $B$. As a result, the proof of \eqref{B'2}, the following, is identical to that of \eqref{A'2}, and so we omit it.

\sta{\label{B'2} If algorithm \textbf{\textsf{A}} terminates at step \ref{s7}, then $(G^*,L^*)$ represents $(G,L)$.}

Then we continue with the following.

\sta{\label{bad2A} If algorithm \textbf{\textsf{A}} terminates at step \ref{s8}, then $(G^*,L^*)$ represents $(G,L)$.}

	The `if condition' of step \ref{s8} is satisfied, and so there exist distinct vertices $a_1, a_2\in A$ with $L(a_1)=L(a_2)$ and $|L(a_1)|=|L(a_2)|=2$. Also, we have $G^*=G$, $L^*(a')=L(a')\setminus \{4\}$ for every $a'\in A'$ and $L^*(v)=L(v)$ for every $v\in V(G)\setminus A'$. 
	
	By \eqref{getbullets}, all seven bullets of Lemma \ref{ABhomo} hold. In particular, by the first dash of the seventh bullet of Lemma \ref{ABhomo}, we have $|A'|=1$, say $A'=\{a'_0\}$, and by the second bullet of Lemma \ref{ABhomo}, we have $L(a'_0)=\{4,5\}$. As a result, we have $4\in L(a'_0)\setminus L^*(a'_0)$, which in turn implies that $|L^*(a'_0)|<|L(a'_0)|$. So $(G^*,L^*)$ satisfies (R\ref{r1}).
	
	To argue the validity of (R\ref{r2}), suppose that $G$ admits a frugal $L$-coloring $\phi$. Note that by the fourth bullet of Lemma \ref{ABhomo}, $a_1$ and $a_2$ are adjacent in $G$. So from $L(a_1)=L(a_2)$, $|L(a_1)|=|L(a_2)|=2$ and $4\in L(a_1)\cap L(a_2)$,  we have $4\in \{\phi(a_1),\phi(a_2)\}$. Also, by the fifth bullet of Lemma \ref{ABhomo}, $a'_0$ is adjacent to both $a_1$ and $a_2$. Therefore, we have  $\phi(a'_0)\in L(a'_0)\setminus \{\phi(a_1),\phi(a_2)\}\subseteq L(a'_0)\setminus \{4\}=L^*(a_0)$. In other words, for every $v\in A'=\{a'_0\}$, we have $\phi(v)\in L^*(v)$.  Moreover, for every $v\in V(G^*)\setminus A'$, we have $\phi(v)\in L(v)=L^*(v)$. Hence, $\phi$ being a frugal $L$-coloring of $G$, by Lemma \ref{frugalspread}, $\phi|_{V(G^*)}$ is a frugal $L^*$-coloring of $G^*$, and so $(G^*,L^*)$ satisfies (R\ref{r2}). Finally, note that $(G^*,L^*)$ is spanning $(G,L)$-refinement, and so by Lemma \ref{spanspread}, $(G^*,L^*)$ satisfies (R\ref{r3}). This proves \eqref{bad2A}.\vsp
	
	Again, we observe that steps \ref{s8} and \ref{s9} of algorithm \textbf{\textsf{A}} are symmetrical with respect to $A$ and $B$. For this reason, the proof of \eqref{bad2B} below is identical to that of \eqref{bad2A}, and so we omit it.
	
\sta{\label{bad2B} If algorithm \textbf{\textsf{A}} terminates at step \ref{s9}, then $(G^*,L^*)$ represents $(G,L)$.}
\sta{\label{u_0deg4done} If algorithm \textbf{\textsf{A}} terminates at step \ref{s10}, then $(G^*,L^*)$ represents $(G,L)$.}

The `if condition' of step \ref{s10} is satisfied, that is $|N_{G^L}(u_0)|\geq 4$. Also, since the algorithm \textbf{\textsf{A}} does not terminate at step \ref{s3}, the `if condition' of step \ref{s3} does not hold. In particular, we have $|N_{G^L}(u_0)|\leq 4$, and so $|N_{G^L}(u_0)|=4$. In addition, we have $G^*=G$, $L^*(a')=L(a')\setminus \{4\}$ for every $a\in A'$, $L^*(b')=L(b')\setminus \{5\}$ for every $b'\in B'$ and $L^*(v)=L(v)$ for every $v\in V(G)\setminus (A'\cup B')$.

By \eqref{getbullets}, all seven bullets of Lemma \ref{ABhomo} hold. In particular, by the first dash of the seventh bullet of Lemma \ref{ABhomo}, we have $A'\cap B'=\emptyset$ and $|A'|=|B'|=1$, say $A'=\{a'_0\}$ and $B'=\{b'_0\}$ for distinct $a'_0,b'_0$, and by the second bullet of Lemma \ref{ABhomo}, we have $L(a'_0)=L(b'_0)=\{4,5\}$. As a result, we have $4\in L^*(a'_0)\setminus L(a'_0)$, which in turn implies that $|L^*(a'_0)|<|L(a'_0)|$. Thus, $(G^*,L^*)$ satisfies (R\ref{r1}).

To see (R\ref{r2}), suppose that $G$ admits a frugal $L$-coloring $\phi$.
Since $|N_{G^L}(u_0)|=4$, by Lemma \ref{smallgooddeg}, we have $\phi(N_{G^L}[u_0])=[5]$. Thus, from $L(u_0)=\{1,2,3\}$ and the definition of $A$ and $B$, we deduce that there exists $a_0\in A$ with $\phi(a_0)=4$ and $b_0\in B$ with $\phi(b_0)=5$.  On the other hand, by the fifth bullet of Lemma \ref{ABhomo}, $a'_0$ is adjacent to $a_0$ and $b'_0$ is adjacent to $b_0$ in $G^L$ (and so in $G$). Therefore, we have  $\phi(a'_0)\in L(a'_0)\setminus \{\phi(a_0)\}=L(a'_0)\setminus \{4\}=L^*(a'_0)$ and $\phi(b'_0)\in L(b'_0)\setminus \{\phi(b_0)\}=L(b'_0)\setminus \{5\}=L^*(b'_0)$. In other words, for every $v\in \{a'_0,b'_0\}=A'\cup B'$, we have $\phi(v)\in L^*(v)$. Moreover, for every $v\in V(G^*)\setminus (A'\cup B')$, we have $\phi(v)\in L(v)=L^*(v)$. Hence, $\phi$ being a frugal $L$-coloring of $G$, by Lemma \ref{frugalspread}, $\phi|_{V(G^*)}$ is a frugal $L^*$-coloring of $G^*$, and so  and so $(G^*,L^*)$ satisfies (R\ref{r2}). Finally, note that $(G^*,L^*)$ is spanning $(G,L)$-refinement, and so by Lemma \ref{spanspread}, $(G^*,L^*)$ satisfies (R\ref{r3}). This proves \eqref{u_0deg4done}.\vsp

From \eqref{geq5done}-\eqref{u_0deg4done}, we deduce:

\sta{\label{finalstep} The output $(G^*,L^*)$ of algorithm \textbf{\textsf{A}} represents  $(G,L)$.}
If algorithm \textbf{\textsf{A}} terminates at one of the steps \ref{s3}-\ref{s10}, then by \eqref{geq5done}-\eqref{u_0deg4done}, we are done. Therefore, we may assume that algorithm \textbf{\textsf{A}} stops at step \ref{s11}. Since algorithm \textbf{\textsf{A}} does not terminates at steps \ref{s3} and \ref{s10}, the `if condition' in these two steps is not satisfied, and so $3\leq |L(u_0)|\leq |N_{G^L}(u_0)|\leq 3$; that is, $|N_{G^L}(u_0)|=|L(u_0)|=3$ and $L(u_0)=\{1,2,3\}$. By \eqref{getbullets}, all seven bullets of Lemma \ref{ABhomo} hold.  In particular, by the first and the second dash of the seventh bullet of Lemma \ref{ABhomo}, we have $A'\cap B'=\emptyset$ and $|A'|=|B'|=1$, say $A'=\{a'_0\}$ and $B'=\{b'_0\}$ for distinct $a'_0,b'_0$, $A,B\neq \emptyset$ and $A\cap B=\emptyset$. Also, by the fifth bullet and the third dash of the second bullet of Lemma \ref{ABhomo}, $a'_0$ is complete to $A$ in $G^L$ (and so in $G$)  and anticomplete to $B$ in $G$ (and so in $G^L$), and $b'_0$ is complete to $B$ in $G^L$ (and so in $G$)  and anticomplete to $A$ in $G$ (and so in $G^L$). Moreover, by the second bullet of Lemma \ref{ABhomo}, we have $L(a'_0)=L(b'_0)=\{4,5\}$.

Let $M$, $i$ and $j$ be as in step \ref{s11} of algorithm \textbf{\textsf{A}}. Then we have $~{G^*=G|(\{u_0,a,b\}\cup (V(G)\setminus N_{G^L}[u_0]))}$, $L^*(u_0)=M$, $L^*(a)=\{i,4\}$, $L^*(b)=\{j,5\}$ and $L^*(v)=L(v)$ for every $v\in V(G^*)\setminus \{u_0,a,b\}$. Also, by the minimality of $M$, we have $M=\{i,j\}\subseteq \{1,2,3\}$. 

To verify validity of (R\ref{r1}), note that from $|N_{G^L}(u_0)|=3$, one may deduce $~|{V(G^*)|=3+|V(G)\setminus N_{G^L}[u_0]|<4+|V(G)\setminus N_{G^L}[u_0]|\leq |V(G)|}$. So $(G^*,L^*)$ satisfies (R\ref{r1}).

For (R\ref{r2}), suppose that $G$ admits a frugal $L$-coloring $\phi$. From $L(u_0)=\{1,2,3\}$, $|N_{G^L}(u_0)|=3$ and Lemma \ref{smallgooddeg}, we observe that either there exists $a_0\in A$ with $\phi(a_0)=4$ or there exists $b_0\in B$ with $\phi(b_0)=5$. On the other hand, by the fifth bullet of Lemma \ref{ABhomo}, $a'_0$ is complete to $A$ in $G^L$ (and so in $G$) and $b'_0$ is complete to $B$ in $G^L$ (and so in $G$). Consequently, since $\phi$ is proper, either $\phi(a'_0)=5$ or $\phi(b'_0)=4$. In the former case, let $\psi(a)=4$, $\psi(b)=j$, $\psi(u_0)=i$ and $\psi(v)=\phi(v)$ for every $v\in V(G^*)\setminus \{u_0,a,b\}=V(G)\setminus N_{G^L}[u_0]$. In the latter case, let $\psi(a)=i$, $\psi(b)=5$, $\psi(u_0)=j$ and again $\psi(v)=\phi(v)$ for every $v\in V(G^*)\setminus \{u_0,a,b\}=V(G)\setminus N_{G^L}[u_0]$. We leave the reader to check that, from $\phi$ being a frugal $L$-coloring of $G$, it follows that $\psi$ is a frugal $L^*$-coloring of $G^*$. This verifies (R\ref{r2}) for $(G^*,L^*)$.

It remains to argue the truth of (R\ref{r3}) for $(G^*,L^*)$. Suppose $\psi$ is an $L^*$-coloring of $G^*$.  Due to $|N_{G^L}(u_0)|=3$, let $N_{G^L}(u_0)\setminus \{a,b\}=\{c\}$. By the first bullet of Lemma \ref{ABhomo}, we have $|L(a)|,|L(b)|,|L(c)|\geq 2$. Note that either $\psi(a'_0)=5$ or $\psi(b'_0)=4$, for otherwise from $\psi(a'_0)=4$ and $\psi(b'_0)=5$, it follows that $\psi(a)=i$, $\psi(b)=j$, and so $\psi(u_0)\in M=\{i,j\}= \{\psi(a),\psi(b)\}$, which contradicts $\psi$ being proper. We deduce (R\ref{r3}) for cases $\psi(a'_0)=\psi(b'_0)$ and $\psi(a'_0)\neq \psi(b'_0)$ separately, below.

First, suppose that $\psi(a'_0)=\psi(b'_0)$, and by symmetry, let $\psi(a'_0)=\psi(b'_0)=4$. We define a coloring $\phi$ of $G$ as follows. Let $\phi(b)=5$ and $\phi(v)=\psi(v) $ for every $v\in V(G)\setminus N_{G^L}[u_0]$. In order to determine $\phi(a)$ and $\phi(c)$, we need to consider two cases. If $c\in A$, then from $A\cap B=\emptyset$, we have $5\notin L(a)\cup L(c)$, and so we may choose two distinct colors $k$ and $l$ with $~{k\in L(a)\setminus \{4,5\}=L(a)\setminus \{\phi(a'_0),\phi(b'_0),\phi(b)\}}$ and $l\in L(c)\setminus \{4\}=L(c)\setminus \{\phi(a'_0),\phi(b'_0),\phi(b)\}$, since otherwise $L(a)=L(c)$ and $|L(a)|=|L(c)|=2$, and so algorithm \textbf{\textsf{A}} should have terminated at step \ref{s8}, a contradiction. Otherwise, if $c\notin A$, the we have $4\notin L(c)$. So since $|L(a)|\geq 2$, there exists $k\in L(a)\setminus \{4,5\}=L(a)\setminus \{\phi(a'_0),\phi(b'_0),\phi(b)\}$, and since $|L(c)|\geq 2$, there exists $l\in L(c)\setminus \{k,5\}=L(c)\setminus \{k,4,5\}=L(c)\setminus \{k,\phi(a'_0),\phi(b'_0),\phi(b)\}$, as otherwise $c\in B$ and $L(a)\cap L(c)\neq \emptyset$, which violates the fourth dash of the seventh bullet of Lemma \ref{ABhomo}. We define $\phi(a)=k$ and $\phi(c)=l$. Eventually, we choose $\phi(u_0)\in L(u_0)
\setminus \{k,l,4\}=\{1,2,3\}\setminus \{k,l\}$. We leave it to the reader to check that since $\psi$ is and $L^*$-coloring of $G^*$, $\phi$ is an $L$-coloring of $G$, and so the third bullet of \eqref{finalstep} follows. Note that the argument for the case $\psi(a'_0)=\psi(b'_0)=5$ is analogous, with an additional caveat that this time we rely on the fact that algorithm \textbf{\textsf{A}} does not terminate at step \ref{s9}, instead.

Next, suppose that $\psi(a'_0)\neq \psi(b'_0)$; that is, $\psi(a'_0)=5$ and $\psi(b'_0)=4$. We define a coloring $\phi$ of $G$ as follows. Let $\phi(a)=4$, $\phi(b)=5$, $\phi(v)=\psi(v) $ for every $v\in V(G)\setminus N_{G^L}[u_0]$. Also since $A\cap B=\emptyset$, either $4\notin L(c)$ or $5\notin L(c)$, and from $|L(c)|\geq 2$, there exists $k\in L(c)\setminus \{4,5\}=L(c)\setminus \{\phi(a),\phi(b),\psi(a'_0),\psi(b'_0)\}$, and we set $\phi(c)=k$. Finally, we choose $\phi(u_0)\in L(u_0)
\setminus \{k,4,5\}=\{1,2,3\}\setminus \{k\}$. Then it is easy to check that since $\psi$ is an $L^*$-coloring of $G^*$, $\phi$ is an $L$-coloring of $G$, and so $(G^*,L^*)$ satisfies (R\ref{r3}). This proves \eqref{finalstep}.

\sta{\label{23kill1} There exists a $(G^*,L^*)$-refinement $(\tilde{G},\tilde{L})$ with the following specifications.	
	\begin{itemize}
			\item $(\tilde{G},\tilde{L})$ can be computed from $(G^*,L^*)$ is time $\mathcal{O}(|V(G^*)|^2)$.
			\item We have $|\tilde{L}(v)|\neq 1$ for all $v\in V(\tilde{G})$.
		\item If $G^*$ admits a frugal $L^*$-coloring, then $\tilde{G}$ admits a frugal $\tilde{L}$-coloring. 
		\item If $\tilde{G}$ admits an $\tilde{L}$-coloring, then $G^*$ admits an $L^*$-coloring.
\end{itemize}}
We may apply Theorem \ref{kill1} to $(G^*,L^*)$, obtaining a $(G^*,L^*)$-refinement $(\hat{G^*},\hat{L^*})$ satisfying the bullet conditions of Theorem \ref{kill1}. Then, defining $\tilde{G}=\hat{G^*}$ and $\tilde{L}=\hat{L^*}$,  it follows that $(\tilde{G},\tilde{L})$ satisfies the bullet conditions of  \eqref{23kill1}. This proves \eqref{23kill1}.\vsp

To conclude the proof, let $(\tilde{G},\tilde{L})$ be as in \eqref{23kill1}. We show that $(\tilde{G},\tilde{L})$ satisfies Theorem \ref{23}.

By \eqref{23runtime}, algorithm \textbf{\textsf{A}} computes $(G^*,L^*)$ from $(G,L)$ in time $\mathcal{O}(|V(G)|^2)$. Also, by the first bullet of \eqref{23kill1}, $(\tilde{G},\tilde{L})$ can be computed from $(G^*,L^*)$ in time $\mathcal{O}(|V(G^*)|^2)=\mathcal{O}(|V(G)|^2)$. So $(\tilde{G},\tilde{L})$ can be computed from $(G,L)$ in time $\mathcal{O}(|V(G)|^2)$; that is, $(\tilde{G},\tilde{L})$ satisfies the first bullet of Theorem \ref{23}.

For the second bullet of Theorem \ref{23}, we argue the validity of (R\ref{r1}), (R\ref{r2}) and (R\ref{r3}) for $(\tilde{G},\tilde{L})$ separately. By \eqref{finalstep}, $(G^*,L^*)$ satisfies (R\ref{r1}), and so being a $(G^*,L^*)$-refinement, it follows that $(\tilde{G},\tilde{L})$ satisfies (R\ref{r1}), as well.

For (R\ref{r2}) suppose that $G$ admits a frugal $L$-coloring. Then by \eqref{finalstep}, $(G^*,L^*)$ satisfies (R\ref{r2}), and so $G^*$ admits a frugal $L^*$-coloring. Therefore, by the third bullet of \eqref{23kill1}, $\tilde{G}$ admits a frugal $\tilde{L}$-coloring, and so $(\tilde{G},\tilde{L})$ satisfies (R\ref{r2}). 

Finally, for (R\ref{r3}), suppose that $\tilde{G}$ admits an $\tilde{L}$-coloring. Then by the fourth bullet of \eqref{23kill1}, $G^*$ admits an $L^*$-coloring. Also, by \eqref{finalstep}, $(G^*,L^*)$ satisfies (R\ref{r3}), and so $G$ admits an $L$-coloring. Hence, $(\tilde{G},\tilde{L})$ satisfies (R\ref{r3}). This completes the proof.
\end{proof}

\section{Proof of Theorem \ref{thm:main}}\label{mainsec}

In this section, we combine Theorems \ref{frugality}, \ref{goodp3theorem} and \ref{23} to deduce Theorem \ref{thm:main}. First, Theorems \ref{frugality} and \ref{goodp3theorem} are applied to deduce the following.
\begin{theorem}\label{allfromfrug+good}
	For all fixed $k,r\in \mathbb{N}$, there exists $\eta(k,r)\in \mathbb{N}$ with the following property. Let $(G,L)$ be an instance of the \textsc{List-$k$-Coloring Problem} where $G$ is $rP_3$-free graph. Then there exists a $(G,L)$-profile $\Xi(G,L)$ with the following specifications.
	\begin{itemize}
		\item\sloppy $|\Xi(G,L)|\leq \mathcal{O}\left(|V(G)|^{\eta(k,r)}\right)$ and $\Xi(G,L)$ can be computed from $(G,L)$ in time $ \mathcal{O}\left(|V(G)|^{\eta(k,r)}\right)$.
		\item For every $(G',L')\in \Xi(G,L)$ and every $v\in V(G')$, we have $|L'(v)|\neq 1$.
		\item For every $(G',L')\in \Xi(G,L)$, $G'$ has no $L'$-good $P_3$.
		\item If $G$ admits an $L$-coloring, then for some $(G',L')\in \Xi(G,L)$, $G'$ admits a frugal $L'$-coloring. 
		\item If $G'$ admits an $L'$-coloring for some $(G',L')\in \Xi(G,L)$, then $G$ admits an $L$-coloring. 
	\end{itemize}
\end{theorem} 
\begin{proof}
	Applying Theorem \ref{frugality} to $(G,L)$, we obtain a spanning $(G,L)$-profile $\Pi(G,L)$ satisfying the bullet conditions of Theorem \ref{frugality}. Also, for every $(G,K)\in \Pi(G,L)$, applying Theorem \ref{goodp3theorem} to $(G,K)$, we obtain a spanning $(G,K)$-profile $\Upsilon(G,L)$ satisfying the bullet conditions of Theorem \ref{goodp3theorem}. Let $\Theta(G,L)=\bigcup_{(G,K)\in \Pi(G,L)}\Upsilon(G,K)$. Then, for every $(G,J)\in \Theta(G,L)$, we may apply Theorem \ref{kill1} to $(G,J)$, obtaining a $(G,J)$-refinement $(\hat{G},\hat{J})$ satisfying bullet conditions of Theorem \ref{kill1}.
	
	Let $\Xi(G,L)=\{(\hat{G},\hat{J}): (G,J)\in \Theta(G,L)\}$. We claim that $\Xi(G,L)$ satisfies Theorem \ref{allfromfrug+good}. Clearly, $\Xi(G,L)$ is a $(G,L)$-profile. Also, let $\pi(k,r)$ be as in Theorem \ref{frugality} and $\upsilon(k,r)$ be as in Theorem \ref{goodp3theorem}. Now, assuming $\eta(k,r)=\pi(k,r)+\upsilon(k,r)+2$, by the first bullet of Theorems \ref{frugality} and \ref{goodp3theorem} and \ref{kill1}, $\Xi(G,L)$ satisfies the first bullet of Theorem \ref{allfromfrug+good}. Also, by the second bullet of Theorem \ref{kill1}, $\Xi(G,L)$ satisfies the second bullet of Theorem \ref{allfromfrug+good}. Moreover, the second bullet of Theorem \ref{goodp3theorem} along with Lemma \ref{goodspread} implies that $\Xi(G,L)$ satisfies the third bullet of Theorem \ref{allfromfrug+good}. The fourth bullet of Theorem \ref{allfromfrug+good} for $\Xi(G,L)$ follows from the second bullet of Theorem \ref{frugality} and the third bullets of Theorems \ref{goodp3theorem} and \ref{kill1}. Finally, by Lemma \ref{spanspread} and the fourth bullet of Theorem \ref{kill1}, $\Xi(G,L)$ satisfies the fifth bullet of Theorem \ref{allfromfrug+good}. This completes the proof.
\end{proof}

\sloppy Next, we prove the following as an application of Theorem \ref{23}. Recall the definition $~{p(G,L)=|V(G)|+\sum_{v\in V(G)}|L(v)|}$ for every instance $(G,L)$ of the \textsc{List-$k$-Coloring Problem}, $k\in \mathbb{N}$.

\begin{theorem}\label{23rep}
	Let  $(G,L)$ be an instance of the \textsc{List-$5$-Coloring Problem} such that $|L(v)|\neq 1$ for all $v\in V(G)$ and $G$ has no $L$-good $P_3$. Then there exists a $(G,L)$-refinement $(G^{\flat},L^{\flat})$ with the following specifications. 
	\begin{itemize}
		\item $(G^{\flat},L^{\flat})$ can be computed from $(G,L)$ in time $\mathcal{O}(p(G,L)|V(G)|^2)=\mathcal{O}(|V(G)|^3)$.
		\item $|L^{\flat}(v)|\in \{0,2\}$ for all $v\in V(G^{\flat})$.
		\item If $G$ admits a frugal $L$-coloring, then $G^{\flat}$ admits a frugal $L^{\flat}$-coloring. 
		\item If $G^{\flat}$ admits an $L^{\flat}$-coloring, then $G$ admits an $L$-coloring.
	\end{itemize}
\end{theorem}
\begin{proof}
	Let $(G,L)$ be a counterexample with $p(G,L)$ as small as possible. If $|L(u)|\in \{0,2\}$ for every $u\in V(G)$, then we define $G^{\flat}=G$, $L^{\flat}=L$, and it is immediately seen that $(G^{\flat},L^{\flat})$ satisfies the bullet conditions of Theorem \ref{23rep}, a contradiction.  So we may assume that there exists a vertex $u_0\in V(G)$ with $|L(u_0)|\geq 3$. Applying Theorem \ref{23} to $(G,L)$ and $u_0$, we obtain a $(G,L)$-refinement $(\tilde{G},\tilde{L})$, satisfying the bullet conditions of Theorem \ref{23}. In particular, by the second bullet of Theorem \ref{23}, we have $|\tilde{L}(v)|\neq 1$ for all $v\in V(\tilde{G})$. Also, since $G$ has no $L$-good $P_3$, by Lemma \ref{goodspread}, $\tilde{G}$ has no $\tilde{L}$-good $P_3$. Moreover, by the third bullet of Theorem \ref{23}, $(\tilde{G},\tilde{L})$ satisfies (R\ref{r1}); that is, $p(\tilde{G},\tilde{L})<p(G,L)$. This, together with the minimality of $p(G,L)$, implies that there exists a $(\tilde{G},\tilde{L})$-refinement $(\tilde{G}^{\flat},\tilde{L}^{\flat})$, satisfying the bullet conditions of Theorem \ref{23rep}. Now, let $G^{\flat}=\tilde{G}^{\flat}$ and  $L^{\flat}=\tilde{L}^{\flat}$. Since $(\tilde{G}^{\flat},\tilde{L}^{\flat})$ is a   $(\tilde{G},\tilde{L})$-refinement and $(\tilde{G},\tilde{L})$ is a $(G,L)$-refinement, it follows that $(G^{\flat},L^{\flat})=(\tilde{G}^{\flat},\tilde{L}^{\flat})$ is a $(G,L)$-refinement. Moreover, since $(G^{\flat},L^{\flat})=(\tilde{G}^{\flat},\tilde{L}^{\flat})$, it is easy to see that
	\begin{itemize}
	    \item[-] the first bullet of Theorem \ref{23rep} for $(G,L)$ and $(G^{\flat},L^{\flat})$ follows from the first bullet of Theorem \ref{23} for $(G,L)$ and $(\tilde{G},\tilde{L})$ and the first bullet of Theorem \ref{23rep} for $(\tilde{G},\tilde{L})$ and $(\tilde{G}^{\flat},\tilde{L}^{\flat})$;
	    \item[-] the second bullet of Theorem \ref{23rep} for $(G^{\flat},L^{\flat})$ follows from the second bullet of Theorem \ref{23rep} for $(\tilde{G}^{\flat},\tilde{L}^{\flat})$;
	    \item[-] the third bullet of Theorem \ref{23rep} for $(G,L)$ and $(G^{\flat},L^{\flat})$ follows from the third bullet of Theorem \ref{23} (in particular, (R\ref{r2})) for $(G,L)$ and $(\tilde{G},\tilde{L})$ and the third bullet of Theorem \ref{23rep} for $(\tilde{G},\tilde{L})$ and $(\tilde{G}^{\flat},\tilde{L}^{\flat})$; and 
	    \item[-] the fourth bullet of Theorem \ref{23rep} for $(G,L)$ and $(G^{\flat},L^{\flat})$ follows from the third bullet of Theorem \ref{23} (in particular, (R\ref{r3})) for $(G,L)$ and $(\tilde{G},\tilde{L})$ and the fourth bullet of Theorem \ref{23rep} for $(\tilde{G},\tilde{L})$ and $(\tilde{G}^{\flat},\tilde{L}^{\flat})$.
	\end{itemize}
	But this violates $(G,L)$ being a counterexample to Theorem \ref{23rep}, and so completes the proof.
\end{proof}
As the last ingredient, we need the following, which is proved via a reduction to 2SAT, and has been discovered independently by many authors \cite{2sat1,2sat2,2sat3}.
\begin{theorem}[Edwards \cite{2sat1}]\label{2LC}
Let $k\in \mathbb{N}$ be fixed and  $(G,L)$ be an instance of the \textsc{List-$k$-Coloring Problem} with $|L(v)|\leq 2$ for every $v\in V(G)$. Then it can be decided in time $\mathcal{O}(|V(G)|^2)$ whether $G$ admits an $L$-coloring.
\end{theorem}
Now we are in a position to prove Theorem \ref{thm:main}, which we restate.
\begin{theorem}
Let $r\in \mathbb{N}$ be fixed. Then there exists a polynomial-time algorithm which solves the \textsc{List-$5$-Coloring Problem} restricted to $rP_3$-free instances.
\end{theorem}
	\begin{proof}
	Given an $rP_3$-free instance $(G,L)$ of the \textsc{List-$5$-Coloring Problem}, let $\Xi(G,L)$  be as in Theorem \ref{allfromfrug+good}. Then, for every $(G',L')\in \Xi(G,L)$, by the seond bullet of Theorem \ref{allfromfrug+good}, $|L'(v)|\neq 1$ for all $v\in V(G)$, and by the third bullet of Theorem \ref{allfromfrug+good}, $G'$ has no $L$-good $P_3$. Therefore, we may apply Theorem \ref{23rep} to $(G',L')$ obtaining a $(G,L)$-refinement $(G'^{\flat},L'^{\flat})$ satisfying the bullet conditions of Theorem \ref{23rep}. Now, consider the $(G,L)$-profile $\Gamma(G,L)=\{(G'^{\flat},L'^{\flat}): (G',L')\in \Xi(G,L)\}$. For all $k,r\in \mathbb{N}$, let $\eta(k,r)$ be as in Theorem \ref{allfromfrug+good}. Then, statement \eqref{finalruntime} below follows immediately from the first bullet of Theorem \ref{allfromfrug+good} for $\Xi(G,L)$, and the first bullet of Theorem \ref{23rep} for every $(G'^{\flat},L'^{\flat})$, where $(G',L')\in \Xi(G,L)$.
	
	\sta{\label{finalruntime} $|\Gamma(G,L)|\leq \mathcal{O}(|V(G)|^{\eta(k,r)})$, and  $\Gamma(G,L)$ can be computed from $(G,L)$ in time $\mathcal{O}(|V(G)|^{\eta(k,r)+3})$.}
	
	Also, we deduce: 
	
	\sta{\label{ordtoord} $G$ admits an $L$-coloring if and only if there exists $(G'^{\flat},L'^{\flat})\in \Gamma(G,L)$ such that $G'^{\flat}$ admits an $L'^{\flat}$-coloring for some $(G',L')\in \Xi(G,L)$.}
	
	Suppose that $G$ admits an $L$-coloring. By the fourth bullet of Theorem \ref{allfromfrug+good}, for some $(G',L')\in \Xi(G,L)$, $G'$ admits a frugal $L'$-coloring. As a result, by the third bullet of Theorem \ref{23rep}, $(G'^{\flat},L'^{\flat})\in \Gamma(G,L)$ admits a frugal $L'^{\flat}$-coloring, and so an $L'^{\flat}$-coloring.
	
	Conversely, suppose that for some $(G',L')\in \Xi(G,L)$,  $G'^{\flat}$ admits an $L'^{\flat}$-coloring. Then by the fourth bullet of Theorem \ref{23rep}, $G'$ admits an $L'$-coloring. Therefore, by the fifth bullet of Theorem \ref{allfromfrug+good}, $G$ admits an $L$-coloring. This proves \eqref{ordtoord}.\vsp
	
	Now, the algorithm is as follows. First, we compute $\Gamma(G,L)$. By \eqref{finalruntime}, this is doable in time  $\mathcal{O}(|V(G)|^{\max\{\eta(k,r),3\}})$. Then, by the second bullet of Theorem \ref{23rep}, for each $(G'^{\flat},L'^{\flat})\in \Gamma(G,L)$, we have $|L'^{\flat}(v)|\in \{0,2\}$. Therefore, since $|\Gamma(G,L)|\leq \mathcal{O}(|V(G)|^{\eta(k,r)})$ by \eqref{finalruntime}, applying the algorithm from Theorem \ref{2LC}, we decide in polynomial time whether there exists $(G'^{\flat},L'^{\flat})\in \Gamma(G,L)$ such that $G'^{\flat}$ admits an $L'^{\flat}$-coloring. If the answer is yes, then by \eqref{ordtoord}, $G$ admits an $L$-coloring. If the answer is no, again by \eqref{ordtoord}, $G$ admits no $L$-coloring. This completes the proof.
	\end{proof}
\section{Proof of Theorem \ref{thm:hardness}} \label{sec:hardness}
In this section, we prove Theorem \ref{thm:hardness} via a reduction from \textit{monotone} \textsc{NAE3SAT}, defined as follows. The \textsc{Not-All-Equal-3-Satisfiability Problem (NAE3SAT)} is to decide, given an instance $I$ consisting of $n$ Boolean variables $x_1,\ldots,x_n$ and $m$ clauses $C_1,\ldots, C_m$, each containing three literals, whether there exists a true/false assignment for each variable such that each clause contains at least one true literal and one false literal. We say $I$ is \textit{satisfiable} if it admits such an assignment. By \emph{monotone} \textsc{NAE3SAT}, we mean \textsc{NAE3SAT} restricted to \textit{monotone} instances; that is, instances with no negated literals.

\begin{theorem}[Garey and Johnson \cite{NAE-NP}]\label{NAENP}
	Monotone \textsc{NAE3SAT} is \textsf{NP}-complete.
\end{theorem}
Now we can prove Theorem \ref{thm:hardness}, which we restate. 

\begin{theorem}
	The \textsc{$k$-Coloring Problem} restricted to $2P_4$-free graphs (and hence $rP_4$-free graphs for every fixed $r\geq 2$) is \textsf{NP}-complete for all $k\geq 5$ and $r\geq 2$.
\end{theorem}
\begin{proof}
Clearly, the \textsc{$k$-Coloring Problem} restricted to $2P_4$-free graphs belongs to \textsf{NP}. For the hardness, given a monotone \textsc{NAE3SAT} instance $I$ with variables $x_1,x_2,...,x_n$ and clauses $C_1,C_2,...,C_m$, we construct a graph $G$ as an instance of the \textsc{$5$-coloring Problem}, as follows. Let $C=\{c_i: i\in [5]\}$, $X=\{x_i:i\in [n]\}$, $Y=\{y_j,z_j:j\in [m]\}$ and $U=\{u_j^k,w_j^k:j\in [m],k\in [3]\}$. Then we let $V(G)=C\cup X\cup Y\cup U$, and the adjacency in $G$ is as follows.
\begin{itemize}
	\item[-] $C$ is a clique of $G$.
	\item[-] $\{c_3,c_4,c_5\}$ is complete to $X$.
	\item[-] $\{c_1,c_2\}$ is complete to $Y$.
	\item[-] For each $j\in [m]$ and $k\in[3]$, we have $c_1u_j^k,c_2w_j^k\in E(G)$. 
	\item[-] For each $j \in [m]$, we have $c_iu_j^k,c_iw_j^k\in E(G)$ for all pairs $(i, k)$ with $i \in \{3,4,5\}$, $k \in \{1,2,3\}$, and $i \neq k+2$. 
	\item[-] $X$ is complete to $Y$.
	\item[-] For each $j\in [m]$ and all $k\in[3]$, we have $y_ju_j^k,z_jw_j^k\in E(G)$.
	\item[-] For each $j\in [m]$, if $C_j$ contains $x_{i_1}$, $x_{i_2}$, and $x_{i_3}$, then we have $x_{i_k}u_j^k,x_{i_k}w_j^k\in E(G)$ for all $k\in [3]$.
\end{itemize}
There are no edges in $E(G)$ other than those described above. It is easily seen that the construction is of polynomial size and can be computed in polynomial time.

\sta{\label{feas} $I$ is satisfiable if and only if $G$ is $5$-colorable.}
First, let $\phi:V(G)\rightarrow [5]$ be a $5$-coloring of $G$. Since $C$ is a clique of $G$, we may assume without loss of generality that that $\phi(c_i)=i$ for every $i\in[5]$. Thus, $\phi(x_i)\in \{1,2\}$ for every $i\in [n]$. Now, let $j\in [m]$ and let $x_{i_1}$, $x_{i_2}$, and $x_{i_3}$ be the literals in $C_j$. If $\phi(x_{i_1})=\phi(x_{i_2})=\phi(x_{i_3})=2$, then we have $\phi(u_j^1)=3$, $\phi(u_j^2)=4$ and $\phi(u_j^3)=5$. But then $y_j$ has a neighbour of each color in $[5]$, which is a contradiction. As a result, at least one of $\phi(x_{i_1})$, $\phi(x_{i_2})$ and $\phi(x_{i_3})$ is equal to $1$. Similarly, by considering the vertex $w_j^k$ for $k\in \{1,2,3\}$, we deduce that at least one of $\phi(x_{i_1})$, $\phi(x_{i_2})$ and $\phi(x_{i_3})$ is $2$. Thus, by setting $x_i$ to be True if $\phi(x_i)=1$ and False if $\phi(x_i)=2$, we conclude that $I$ is satisfiable.
	
	Next, suppose that $I$ is satifiable. We define a coloring $\phi:V(G) \rightarrow [5]$ of $G$ as follows. 
	Let $\phi(c_i)=i$ for every $i\in[5]$. 
	Let $\phi(x_i)=1$ if $x_i$ is assigned True and $\phi(x_i)=2$ otherwise. For each $j\in [m]$ and $C_j$ with literals $x_{i_1}$, $x_{i_2}$, and $x_{i_3}$, and each $k\in [3]$, if $\phi(x_{i_k})=1$, then we set $\phi(u_j^k)=2$ and $\phi(w_j^k)=k+2$, and if $\phi(x_{i_k})=2$, we set $\phi(u_j^k)=k+2$ and $\phi(w_j^k)=1$. Since at least one of $x_{i_1}, x_{i_2}$ and $x_{i_3}$ is assigned True and at least one of them is assigned False, there exists $k_1,k_2\in [3]$ with $k_1\neq k_2$ such that $\phi(u_j^{k_1})=2$ and $\phi(w_j^{k_2})=1$. We set $\phi(y_j)=k_1+2$ and $\phi(z_j)=k_2+2$. We leave it to the reader to check that $\phi$ is $5$-coloring of $G$. This proves \eqref{feas}.
	
\sta{\label{P4type}
	The vertex set of every induced $P_4$ in $G$ intersects either $C$ or both $X$ and $Y$.}
	Let $P$ be an induced $P_4$ in $G$ with $V(P)=\{v_1,v_2,v_3,v_4\}$ and $E(P)=\{v_1v_2,v_2v_3,v_3v_4\}$ such that $V(P)\cap C=\emptyset$. If $v_2\in U$, then without loss of generality, we may assume that $v_1\in X$ and $v_3\in Y$, as desired. So $v_2\notin U$, and similarly, $v_3\notin U$. It follows that $v_2,v_3\in X\cup Y$. Therefore, since $X$ and $Y$ are stable sets of $G$ and $v_2v_3\in E(G)$, one of $v_2$ and $v_3$ belongs to  $X$ and the other one belongs to $Y$. This proves \eqref{P4type}.

\sta{\label{2P4free} $G$ is $2P_4$-free.}
	Suppose not. Let $P$ and $Q$ be two induced $P_4$'s in $G$ with $V(P)$ anticomplete to $V(Q)$. Since $C$ is a clique of $G$, we may assume without loss of generality that $V(P)\cap C\neq \emptyset$. By \eqref{P4type}, we may choose vertices $x\in V(P)\cap X$ and $y\in V(P)\cap Y$. If there exists $c_i\in V(Q)\cap C$, then depending on whether $i\in \{1,2\}$ or not, either $c_iy\in E(G)$ or $c_ix\in E(G)$, which is impossible. Therefore, by \eqref{P4type}, we may choose a vertex $x'\in V(Q)\cap X$. But then $x'y\in E(G)$, a contradiction. This proves \eqref{2P4free}.\vsp
	
	From \eqref{feas}, \eqref{2P4free} and Theorem \ref{NAENP}, it follows that the \textsc{$5$-Coloring Problem} restricted to $2P_4$-free graphs is \textsf{NP}-hard. This completes the proof.
\end{proof}

\section*{Acknowledgments}
We are thankful to the anonymous referees for carefully reading the paper and suggesting a number of improvements.
	
	\bibliographystyle{abbrv}
	\bibliography{rP3ref}   
\end{document}